\documentclass[10pt]{amsart}
\usepackage{graphicx}
\usepackage{amscd}
\usepackage{amsmath}
\usepackage{amsthm}
\usepackage{amsfonts}
\usepackage{amssymb}
\usepackage{mathrsfs}
\usepackage{bm}
\usepackage{enumerate}
\usepackage{amsrefs}
\usepackage{xcolor}
\usepackage[colorlinks, citecolor=blue, linkcolor=red, pdfstartview=FitB]{hyperref}
\usepackage[latin1]{inputenc}

\usepackage{marginnote}	
\usepackage{soul} 			

\newcommand{\bC}{{\mathbb{C}}}
\newcommand{\bD}{{\mathbb{D}}}

\newcommand{\bF}{{\mathbb{F}}}

\newcommand{\bM}{{\mathbb{M}}}
\newcommand{\bN}{{\mathbb{N}}}

\newcommand{\bT}{{\mathbb{T}}}

  \newcommand{\A}{{\mathcal{A}}}
  \newcommand{\B}{{\mathcal{B}}}
  
  \newcommand{\D}{{\mathcal{D}}}
  \newcommand{\E}{{\mathcal{E}}}
  \newcommand{\F}{{\mathcal{F}}}
  \newcommand{\G}{{\mathcal{G}}}
\renewcommand{\H}{{\mathcal{H}}}
  
  \newcommand{\J}{{\mathcal{J}}}
  \newcommand{\K}{{\mathcal{K}}}  

  \newcommand{\M}{{\mathcal{M}}}
  \newcommand{\N}{{\mathcal{N}}}
\renewcommand{\O}{{\mathcal{O}}}

\renewcommand{\S}{{\mathcal{S}}}
  
  \newcommand{\U}{{\mathcal{U}}}
  
  \newcommand{\W}{{\mathcal{W}}}
  \newcommand{\X}{{\mathcal{X}}}

\newcommand{\fA}{{\mathfrak{A}}}

\newcommand{\fB}{{\mathfrak{B}}}

\newcommand{\fE}{{\mathfrak{E}}}
\newcommand{\fe}{{\mathfrak{e}}}

\newcommand{\fJ}{{\mathfrak{J}}}
\newcommand{\fK}{{\mathfrak{K}}}

\newcommand{\fN}{{\mathfrak{N}}}

\newcommand{\fR}{{\mathfrak{R}}}

\newcommand{\fS}{{\mathfrak{S}}}
\newcommand{\fs}{{\mathfrak{s}}}
\newcommand{\fT}{{\mathfrak{T}}}

\newcommand{\fz}{{\mathfrak{z}}}

\newcommand{\rC}{\mathrm{C}}

\newcommand{\eps}{\varepsilon}
\renewcommand{\phi}{\varphi}

\newcommand{\upchi}{{\raise.35ex\hbox{$\chi$}}}

\newcommand{\ol}{\overline}

\newcommand{\qand}{\quad\text{and}\quad}

\newcommand{\id}{\operatorname{id}}

\newcommand{\ran}{\operatorname{ran}}
\newcommand{\re}{\operatorname{Re}}

\newcommand{\spec}{\operatorname{spec}}

\newcommand{\Prim}{\operatorname{Prim}}

\newtheorem{lemma}{Lemma}[section]
\newtheorem{theorem}[lemma]{Theorem}
\newtheorem{proposition}[lemma]{Proposition}
\newtheorem{corollary}[lemma]{Corollary}

\theoremstyle{definition}

\newtheorem{question}{Question}
\newtheorem{example}{Example}

\date{\today}
\author{Rapha\"el Clou\^atre}

\address{Department of Mathematics, University of Manitoba, Winnipeg, Manitoba, Canada R3T 2N2}

\email{raphael.clouatre@umanitoba.ca\vspace{-2ex}}
\author{Ian Thompson}
\email{thompsoi@myumanitoba.ca\vspace{-2ex}}
\thanks{R.C. was partially supported by an NSERC Discovery Grant. I.T. was partially supported by a Manitoba Graduate Scholarship.}

\title[Finite-dimensionality in the NC Choquet boundary]{Finite-dimensionality in the non-commutative Choquet boundary: peaking phenomena and $\rC^*$-liminality}
\begin{document}
\begin{abstract}
We explore the finite-dimensional part of the non-commutative Choquet boundary of an operator algebra. In other words, we seek finite-dimensional boundary representations. Such representations may fail to exist even when the underlying operator algebra is finite-dimensional. Nevertheless, we exhibit mechanisms that detect when a given finite-dimensional representation lies in the Choquet boundary. Broadly speaking, our approach is topological and requires identifying isolated points in the spectrum of the $\rC^*$-envelope. This is accomplished by analyzing peaking representations and peaking projections, both of which being non-commutative versions of the classical notion of a peak point for a function algebra. We also connect this question with the residual finite-dimensionality of the $\rC^*$-envelope and to a stronger property that we call $\rC^*$-liminality. Recent developments in matrix convexity allow us to identify a pivotal intermediate property, whereby every matrix state is locally finite-dimensional.
\end{abstract}
\maketitle

\section{Introduction}\label{S:Introduction}
 
In unraveling the structure of $\rC^*$-algebras, a fruitful paradigm is to model these objects, insofar as possible, by finite-dimensional ones. The idea that finer structural properties of  a large class of $\rC^*$-algebras can be detected upon approximation by matrix algebras has become a major trend in the field. Nuclearity, a  notion at the center of recent capstone results in the classification program for some simple $\rC^*$-algebras (see \cite{GLN2015},\cite{TWW2017} and references therein), is an important example of an ``internal" finite-dimensional approximation property. Closely related to the theme of this paper is a different, ``external" finite-dimensional approximation property which we now describe. 

A $\rC^*$-algebra is said to be \emph{residually finite-dimensional} (or RFD) if it admits a separating set of finite-dimensional $*$-representations.  Roughly speaking, RFD $\rC^*$-algebras should be thought of as being diagonal with finite-dimensional blocks. In particular, they can be completely understood by examining their representations as matrices. Interestingly, various characterizations of this important property have emerged over the years  \cite{EL1992},\cite{archbold1995},\cite{hadwin2014},\cite{CS2019}.

In principle, similar ideas have the potential to unlock the structure of more general, possibly non-selfadjoint, operator algebras. This possibility was explored in \cite{mittal2010} for some concrete operator algebras of functions. In a related context, an analysis of finer forms of residual finite-dimensionality was performed recently in \cite{AHMR2020subh}. Predating this last work, a thorough investigation of non-selfadjoint residual finite-dimensionality was initiated in \cite{CR2018rfd}. Therein, consistently with the more familiar self-adjoint setting, an operator algebra is said to be RFD if it admits a completely norming collection of completely contractive homomorphisms into matrix algebras. Several basic  properties of such algebras were established, such as the non-obvious fact that finite-dimensional algebras are RFD. Furthermore an attempt was made to connect back with the more classical $\rC^*$-algebra setting, via the following procedure. 

Let $\A$ be a unital operator algebra, which we assume is concretely represented on some Hilbert space $\H$, so that $\A\subset B(\H)$. Then one can consider $\rC^*(\A)$, the $\rC^*$-algebra generated by $\A$. Assuming that $\A$ is RFD, it is then natural to ask whether this property is inherited by $\rC^*(\A)$. To make this discussion more efficient, the language of $\rC^*$-covers is useful. Recall that a \emph{$\rC^*$-cover} for $\A$ is a pair $(\fA,\iota)$, where $\fA$ is a $\rC^*$-algebra and $\iota:\A\to \fA$ is a completely isometric homomorphism such that $\fA=\rC^*(\iota(\A))$. If $\A$ is indeed RFD, then essentially by definition there is some $\rC^*$-cover $(\fA,\iota)$ such that $\fA$ is RFD. The question raised above only becomes interesting, then, if the $\rC^*$-cover is fixed in advance. Two natural candidates are the minimal and the maximal $\rC^*$-covers. 

In some special cases, it is known that the maximal $\rC^*$-cover of an RFD operator algebra does inherit the property of being RFD (see \cite[Section 5]{CR2018rfd}). As of this writing, it is still unknown whether this phenomenon always occurs; we will not address this problem here. The corresponding question for the minimal $\rC^*$-cover is the main driving force of the current paper, so we discuss it in more details.

Due to seminal work of Arveson \cite{arveson1969} and Hamana \cite{hamana1979}, any unital operator algebra admits an essentially unique minimal $\rC^*$-cover, which is called its \emph{$\rC^*$-envelope}. The question we are interested in then asks whether an RFD unital operator algebra always admits an RFD $\rC^*$-envelope.  As shown in \cite[Example 4]{CR2018rfd}, the answer is negative even for finite-dimensional operator algebras. Nevertheless, some special cases where the answer is affirmative were identified in that same paper.  The present paper can be viewed as a natural continuation, aiming to clarify this question further.

Our approach is predicated on Arveson's insight for constructing the $\rC^*$-envelope of a concretely represented unital operator algebra $\A\subset B(\H)$. By considering the $\rC^*$-envelope $\rC^*_e(\A)$ as being determined by  a non-commutative analogue of the Shilov boundary of a function algebra, Arveson's vision was that it could be constructed via a non-commutative analogue of the Choquet boundary. This original vision has now been fully realized thanks to non-trivial contributions from many researchers \cite{MS1998},\cite{dritschel2005},\cite{arveson2008},\cite{davidsonkennedy2015}. 

In Arveson's analogy, points are identified with characters in the classical commutative world,  and correspond to irreducible $*$-representations in the non-commu\-tative realm. Thus, the $\rC^*$-envelope is determined by the \emph{boundary representations} (see Subsection \ref{SS:bdryrep} for details.) Therefore, the residual finite-dimensionality of the $\rC^*$-envelope should be detectable through the lens of these boundary representations. For instance, an abundance of finite-dimensional boundary representations is known to force the $\rC^*$-envelope to be RFD. Conversely, the residual finite-dimensionality of the $\rC^*$-envelope implies that, in an appropriate sense, both the finite-dimensional irreducible $*$-representations and the boundary representations are dense. It is unclear, however, whether these two dense sets have any overlap;  even the existence of a single finite-dimensional boundary representation appears to be difficult to ascertain. Our main question is thus the following.

\begin{question}\label{Q:main}
Let $\A$ be a unital operator algebra with an RFD $\rC^*$-envelope. Must the envelope admit a finite-dimensional boundary representation for $\A$?
\end{question}

To give the reader a sense for why this question may be non-trivial, we mention that there are examples of \emph{finite-dimensional} unital operator algebras for which the $\rC^*$-envelope admits no finite-dimensional boundary representations (Example \ref{E:fdimnotNFD}). In fact, it can even happen that the $\rC^*$-envelope of a unital finite-dimensional operator algebra admits simply no finite-dimensional $*$-representations whatsoever (Example \ref{E:Cuntz}). This occurs despite the fact that finite-dimensional operator algebras satisfy a strong form a residual finite-dimensionality: they are in fact normed in finite dimensions \cite[Theorem 3.5]{CR2018rfd} (see Subsection \ref{SS:RFD} below for a definition). Consequently, if the answer to Question \ref{Q:main} is to be affirmative, it must be so for deeper reasons than the mere residual finite-dimensionality of $\A$. We also mention that a boundary representation being finite-dimensional is a much stronger notion than it being ``accessible" in the sense of \cite[Definition 6.32]{kriel2019}.

In view of this difficulty, a strategy for exhibiting finite-dimensional boundary representations must be formulated. We opt for the following topological approach to establish most our main results. As mentioned above, the set of (unitary equivalence classes of) boundary representations is dense in the spectrum of the $\rC^*$-envelope. Thus, if a finite-dimensional irreducible $*$-representation is known to be an isolated point, it would necessarily be a boundary representation.  

 In keeping with the paradigm suggested by Arveson that the $\rC^*$-envelope should be analyzed using non-commutative analogues of ideas from function theory, to realize the aforementioned strategy we view operator algebras as comprising non-commutative functions. Accordingly, we aim to  apply tools from what one could call non-commutative uniform algebra theory.  More precisely, we verify that certain finite-dimensional representations are indeed isolated points by showing that they are ``non-commutative peak points". Classically, given a uniform algebra $\A$ of continuous functions on a compact metric space $X$, a point $\xi\in X$ is said to be a \emph{peak point} for $\A$ if there is $\phi\in \A$ such that
 \[
 \phi(\xi)=1>|\phi(x)|, \quad x\in X, x\neq \xi.
 \]
A theorem of Bishop  implies that the set of peak points coincides with the Choquet boundary of $\A$, which is in turn dense in the Shilov boundary of $\A$ \cite[Section 8]{phelps2001}.

We will exploit two rather different non-commutative versions of peak points. The first one replaces points by irreducible $*$-representations of the $\rC^*$-envelope; this is based on the point of view that points in $X$ should be interpreted as characters. The second version identifies points with their characteristic functions, and thus replaces them by projections lying in the bidual. We will utilize both interpretations.
 
We now describe the organization of the paper and state our main original contributions more precisely.  

 Section \ref{S:prelim} collects various prerequisite material on operator algebras, and proves some elementary facts used throughout.
 
In Section \ref{S:ncpeakpoints}, we tackle Question \ref{Q:main}. Central to our approach and results is the following definition. Let $\A$ is a unital operator algebra and let $\pi$ be an irreducible $*$-representation of its $\rC^*$-envelope $\rC^*_e(\A)$. Let $n\in \bN$ and let  $T\in \bM_n(\rC^*_e(\A))$. Then, $T$ is said to \emph{peak} at $\pi$ if 
 \[
 \|\pi^{(n)}(T)\|>\|\sigma^{(n)}(T)\|
 \]
 for every irreducible $*$-representation $\sigma$ of $\rC^*_e(\A)$ which is not unitarily equivalent to $\pi$. In this case, we also say that $\pi$ is a \emph{peaking representation for $\A$}. 
 
 We introduce a more flexible, ``local" version of this notion. We show that a locally peaking representation for $\A$ must necessarily be a boundary representation (Theorem \ref{T:peakingrep}). It is thus highly desirable to be able to recognize locally peaking representations.  To do this, we leverage the theory of peaking projections developed in a series of papers starting with work of Hay \cite{hay2007},\cite{BHN2008},\cite{read2011},\cite{blecher2013},\cite{clouatre2018lochyp}. 
 
 Let $a\in \A$ be a contraction and consider the support projection $\fs_\pi$ of the irreducible $*$-representation $\pi$ (see Subsection \ref{SS:suppproj}). We say that $a$ \emph{peaks at $\fs_\pi$}  if $a\fs_\pi=\fs_\pi$ and $\|ap\|<1$ for every closed projection $p$ orthogonal to $\fs_\pi$. Upon specializing some known deep results to the separable setting, we show in Corollary \ref{C:peakOAchar} that this occurs if and only if $\fs_\pi$ lies in the weak-$*$ closure of $\A$ inside the bidual of $\rC^*_e(\A)$. We also relate this notion of peaking to the previous one (Theorem \ref{T:peakprojlocpeak}). We summarize these results.
 
 \begin{theorem}\label{T:mainA}
 Let $\A$ be a separable unital operator algebra and let $\pi$ be an irreducible finite-dimensional $*$-representation of $\rC^*_e(\A)$. If $\fs_\pi\in \A^{\perp\perp}\subset \rC^*_e(\A)^{**}$, then $\pi$ is a locally peaking representation for $\A$. In particular, $\pi$ is a boundary representation for $\A$. 
 \end{theorem}
 
 We complement this theorem with some concrete examples where it applies (Example \ref{E:peaking}). For the rest of Section \ref{S:ncpeakpoints}, we turn to a certain uniform version of peaking representations called \emph{strongly peaking} representations; these have been considered by other authors \cite{arveson2011},\cite{NPSV2018},\cite{CDHLZ2019},\cite{DP2020} in different contexts. Motivated by the foregoing discussion, we are interested in strongly peaking representations  for $\A$. Interestingly, this requirement is in fact equivalent to the a priori weaker stipulation of being a strongly peaking representation for the larger algebra $\rC^*_e(\A)$. Furthermore, when $\rC^*_e(\A)$ is assumed to be RFD, strongly peaking representations are automatically finite-dimensional. The following is Corollary \ref{C:peakingRFDfd}.
  
 \begin{theorem}\label{T:mainB}
Let $\A$ be a unital operator algebra such that $\rC_e^*(\A)$ is RFD. Let $\pi$ be a strongly peaking $*$-representation for $\rC^*_e(\A)$. Then, $\pi$ is a finite-dimensional boundary representation for $\A$.
\end{theorem}

Much like for locally peaking representations, support projections can be used to guarantee that a $*$-representation is strongly peaking. Indeed, we show in Lemma \ref{L:clopenproj} that a finite-dimensional irreducible $*$-representation is strongly peaking if its support projection is both closed and open in the sense of Akemann's non-commutative topology (see Subsection \ref{SS:suppproj}). We also identify some sufficient conditions for this property to hold (Theorem \ref{T:specpeak}, Example \ref{E:specpeak}).

Finally, in Section \ref{S:C*lim}, in an attempt to better understand the subtleties inherent to Question \ref{Q:main}, we analyze unital operator algebras whose $\rC^*$-envelope are known to admit many finite-dimensional boundary representations. We start by obtaining a characterization of unital operator algebras admitting an RFD $\rC^*$-envelope; this is accomplished in Theorem \ref{T:ELUEP} with the aid of finite-dimensional approximations of representations, in the spirit of work of Exel--Loring \cite{EL1992}.  We then study unital operator algebras for which \emph{all} boundary representations are finite-dimensional. We call such operator algebras \emph{$\rC^*$-liminal}. The following is Corollary \ref{C:C*lim} and it summarizes our main results on this topic. Notably, the proof uses some recent developments in matrix convexity \cite{HL2019}.
 
 \begin{theorem}\label{T:mainB}
  Let $\A$ be a unital operator algebra. Consider the following statements.
\begin{enumerate}[{\rm (i)}]
\item The algebra $\A$ is $\rC^*$-liminal.
\item Every matrix state of $\A$ can be dilated locally to a finite-dimensional $*$-repre\-sentation of $\rC_e^*(\A)$.
\item The algebra $\rC^*_e(\A)$ is RFD.
\end{enumerate}
Then, we have  {\rm (i)} $\Rightarrow$ {\rm (ii)} $\Rightarrow$ {\rm (iii)}.
 \end{theorem}
 
While we know that  {\rm (iii)} $\not \Rightarrow$ {\rm (ii)}, we do not know if the other implication can be reversed.


\section{Operator algebraic preliminaries}\label{S:prelim}

\subsection{Representations of $\rC^*$-algebras}\label{SS:repC*}
The focal point of this paper is the existence of certain $*$-representations of $\rC^*$-algebras. In this subsection, we briefly recall various equivalence relations and topologies related to these objects. The reader should consult  \cite[Chapter 3]{dixmier1977} for greater detail.

Let $\fA$ be a $\rC^*$-algebra. Recall that a closed two-sided ideal $\fJ\subset \fA$ is said to be \emph{primitive} if it is the kernel of an irreducible $*$-representation. The collection of primitive ideals of $\fA$, denoted by $\Prim(\fA)$, is known as its \emph{primitive ideal space}. It can be endowed with the Jacobson topology. Given a subset $\S\subset \Prim(\fA)$, the closure of $\S$ in this topology is the set of primitive ideals of $\fA$ containing $\bigcap_{\fs\in \S}\fs$. It follows that $\S$ is dense in $\Prim(\fA)$ if and only if $\cap_{\fs\in \S}\fs=\{0\}$. If for each $\fs\in \S$ we choose an irreducible $*$-representation $\pi_\fs$ of $\fA$ with $\fs=\ker \pi_\fs$, then we see that $\S$ is dense if and only if $\bigoplus_{\fs\in \S}\pi_\fs$ is injective.

Given two irreducible $*$-representations $\pi$ and $\sigma$ of $\fA$, we say that $\sigma$ is \emph{weakly contained} in $\pi$ and write $\sigma \prec\pi$ 
if $\ker \pi\subset \ker\sigma$. This is equivalent to the existence of a $*$-homomorphism $\theta:\pi(\fA)\to \sigma(\fA)$ such that $\theta\circ \pi=\sigma$. We say that $\pi$ and $\sigma$ are \emph{weakly equivalent} and write $\pi\sim \sigma$ if they are weakly contained in one another. 
It is well known that finite-dimensional irreducible $*$-representations are minimal in the partial order given by weak containment and in fact display significant rigidity.

\begin{lemma}\label{L:fdimweakcontain}
Let $\pi$ and $\sigma$ be two irreducible $*$-representations of  a $\rC^*$-algebra $\fA$. Assume that $\pi$ is finite-dimensional and that $\sigma \prec \pi$. Then, $\sigma$ is finite-dimensional and unitarily equivalent to $\pi$.
\end{lemma}
\begin{proof}
By assumption, there is a $*$-homomorphism $\theta:\pi(\fA)\to \sigma(\fA)$ such that $\theta\circ \pi=\sigma$. Now, since $\pi$ is irreducible and finite-dimensional, we see that $\pi(\fA)$ is a simple $\rC^*$-algebra, so $\theta$ is a $*$-isomorphism and $\sigma(\fA)$ is a finite-dimensional simple $\rC^*$-algebra. This forces $\theta$ to be implemented by some some unitary equivalence.
\end{proof}

The \emph{spectrum} of $\fA$, denoted by $\widehat\fA$, is the set of unitary equivalence classes of irreducible $*$-representations. Given an irreducible $*$-representation $\pi$, we let $[\pi]\in \widehat\fA$ denote its unitary equivalence class. We can define a topology on $\widehat\fA$ by requiring the natural map
\[
\widehat \fA\to \Prim \fA, \quad [\pi]\mapsto \ker \pi
\]
to be continuous. It follows from this that the closure of the singleton $[\pi]$ in $\widehat\fA$ is the set
\[
\{[\sigma]\in \widehat{\fA}:\sigma\prec \pi\}.
\]

If $\fA$ is a unital $\rC^*$-algebra, then $\S(\fA)$ denotes the state space. The \emph{pure} states are the extreme points of the convex set $\S(\fA)$. Given $\psi\in \S(\fA)$, we let $(\sigma_\psi,\H_\psi,\xi_\psi)$ be the corresponding \emph{GNS representation}. That is to say, $\H_\psi$ is a Hilbert space and $\sigma_\psi:\fA\to B(\H_\psi)$ is a $*$-representation with cyclic vector $\xi_\psi$ such that
\[
\psi(t)=\langle \sigma_\psi(t)\xi_\psi,\xi_\psi \rangle, \quad t\in \fA.
\]
The state $\psi$ is pure if and only if $\sigma_\psi$ is irreducible \cite[Proposition 2.5.4]{dixmier1977}.  If $\Delta\subset \widehat{\fA}$, then we let $\S_\Delta(\fA)$ denote the set of (necessarily pure) states $\psi$ with the property that $[\sigma_\psi]\in \Delta$. This is readily seen to coincide with the collection of states of the form
 \[
 t\mapsto \langle \pi(t)\eta,\eta \rangle, \quad t\in \fA
 \]
 for some irreducible $*$-representation $\pi:\fA\to B(\H_\pi)$ with $[\pi]\in \Delta$ and some unit vector $\eta\in \H_\pi$.

\begin{lemma}{\cite[Theorem 3.4.10]{dixmier1977}}\label{L:densestaterep}
Let $\fA$ be a unital $\rC^*$-algebra and let $\Delta\subset \widehat{\fA}$ be a subset. Then, $\Delta$ is dense in $\widehat{\fA}$ if and only if $\S_\Delta(\fA)$ is weak-$*$ dense in the set of pure states.
\end{lemma}

\subsection{Residually finite-dimensional operator algebras}\label{SS:RFD}
Let $\A$ be an operator algebra. For a positive integer $n$, we let $\bM_n$ denote the $n\times n$ complex matrices. Likewise, $\bM_n(\A)$ denotes the algebra of $n\times n$ matrices with entries in $\A$.  By a \emph{representation} of $\A$, we will always mean a completely contractive homomorphism $\pi:\A\to B(\H_\pi)$ for some Hilbert space $\H_\pi$. For each $n\in \bN$, we let $\pi^{(n)}:\bM_n(\A)\to B(\H^{(n)})$ denote the natural ampliation of $\pi$.

Let $\S$ be a set of representations of $\A$ and let $\M\subset \A$ be a subspace. Then, $\S$ is said to be 
\begin{enumerate}[{\rm (a)}]
\item  \emph{separating for $\M$} if $\bigoplus_{\pi\in \S}\pi$ is injective on $\M$;

\item  \emph{completely norming} for $\M$ if $\bigoplus_{\pi\in \S}\pi$ is completely isometric on $\M$.
\end{enumerate}
Note that if $\M$ is a $\rC^*$-algebra, then $\S$ is separating for $\M$ if and only if it is completely norming for $\M$.

We say that $\A$ is  \emph{residually finite-dimensional} (RFD) if there  is a set of representations of $\A$ on finite-dimensional Hilbert spaces that is completely norming for $\A$. 
We wish to record some known characterizations of RFD $\rC^*$-algebras. For this purpose, we introduce some terminology.
Let $\fA$ be a $\rC^*$-algebra and let $\pi:\fA\to B(\H_\pi)$ be a $*$-representation. A net of $*$-representations
\[
\pi_\lambda:\fA\to B(\H_\pi),\quad \lambda\in \Lambda
\]
is said to be an \emph{approximation} for $\pi$ if 
\[
\lim_\lambda \pi_\lambda(a)\xi=\pi(a)\xi, \quad a\in \fA, \xi\in \H_\pi. 
\]
Here, the limit is taken with respect to the norm topology on $\H_\pi$.
If in addition the space $\pi_\lambda(\fA)\H_\pi$ is finite-dimensional for each $\lambda\in \Lambda$, then the net $(\pi_\lambda)$ is said to be a \emph{finite-dimensional approximation} for $\pi$.  Note that in this case, because $\pi_\lambda(\fA)\H_\pi$ is finite-dimensional, it follows that $\pi_\lambda(\fA)$ is a finite-dimensional $\rC^*$-algebra for every $\lambda\in \Lambda$.

\begin{theorem}\label{T:EL}
Let $\fA$ be a $\rC^*$-algebra. Then, the following statements are equivalent. 
\begin{enumerate}[{\rm (i)}]
\item The algebra $\fA$ is RFD.
\item Every $*$-representation of $\fA$ admits a finite-dimensional approximation.
\item Every irreducible $*$-representation of $\fA$ admits a finite-dimensional approximation.
\item The set of unitary equivalence classes of finite-dimensional irreducible $*$-rep\-resentations is dense in $\widehat{\fA}$.
\end{enumerate}
\end{theorem}
\begin{proof}
(i) $\Leftrightarrow$ (ii): This is \cite[Theorem 2.4]{EL1992}.

(ii) $\Rightarrow$ (iii): This is trivial.

(iii) $\Rightarrow$ (i): There is always a set of irreducible $*$-representations of $\fA$ which is separating for $\fA$. It thus follows that there is a set of finite-dimensional $*$-representations of $\fA$ which is also separating for $\fA$, whence $\fA$ is RFD.

(i) $\Leftrightarrow$ (iv): By definition of the topologies involved, we see that the set of unitary equivalence classes of finite-dimensional irreducible $*$-representations of $\fA$ is dense in $\widehat\fA$ if and only if there is a set of finite-dimensional irreducible $*$-representations of $\fA$ which is separating for $\fA$. The latter is clearly equivalent to $\fA$ being RFD.
\end{proof}

Another useful characterization that we require is the following.

\begin{theorem}\label{T:CS}
Let $\fA$ be a $\rC^*$-algebra. Let $\fN\subset \fA$ be the subset consisting of those elements $a\in \fA$ for which there is a finite-dimensional $*$-representation $\pi$ of $\fA$ such that $\|a\|=\|\pi(a)\|$. Then, the following statements hold.
\begin{enumerate}[{\rm (i)}]
\item The algebra $\fA$ is RFD if and only if $\fN$ is dense in $\fA$.
\item All irreducible $*$-representations of $\fA$ are finite-dimensional if and only if $\fN=\fA$. 
\end{enumerate}
\end{theorem}
\begin{proof}
This is \cite[Theorems 3.2 and 4.4]{CS2019}.
\end{proof}

In \cite{CS2019}, the authors call a $\rC^*$-algebra FDI if it satisfies (ii) above. In the unital case, this is the same class as the \emph{liminal} (or \emph{CCR}) $\rC^*$-algebras \cite[Section 4.2]{dixmier1977}.

There is yet another related notion that we require. Let $\A$ be an operator algebra.  Following \cite{CR2018rfd}, we say that $\A$ is \emph{normed in finite-dimensions} (NFD) if for every $n\in\bN$ and every $A\in \bM_n(\A)$ there is a finite-dimensional Hilbert space $\H_\pi$ and a representation $\pi:\A\to B(\H_\pi)$ such that $\|\pi^{(n)}(A)\|=\|A\|$. 
If $\fA$ is a unital $\rC^*$-algebra, then it is well known that the irreducible $*$-representations of the $\rC^*$-algebra of $\bM_n(\fA)$ are all unitarily equivalent to one of the form $\pi^{(n)}$ for some irreducible $*$-representation $\pi$ of $\fA$ (see the last paragraph of \cite[page 485]{hopenwasser1973}). In view of this observation, it is an easy consequence of Theorem \ref{T:CS}(ii) that a unital $\rC^*$-algebra is NFD if and only if it is FDI. 
 
 Finally, we point out that the full group $\rC^*$-algebra of the free group $\bF_2$ is RFD but not liminal, as it is infinite-dimensional and primitive \cite{choi1980}.
  
\subsection{Biduals of operator algebras}\label{SS:biduals}
Some of our techniques will require us to work within biduals of operator algebras, so we recall some relevant basic facts here. 
Let $\fA$ be a $\rC^*$-algebra. Then, the Banach space $\fA^{**}$ can be given the structure of a von Neumann algebra (see for instance \cite[Theorem III.2.4]{takesaki2002} or \cite[Theorem A.5.6]{BLM2004}). Furthermore, if $\A\subset \fA$ is a subalgebra, then $\A^{**}$ can be viewed as subalgebra of $\fA^{**}$ (see the proof of \cite[Corollary 2.5.6]{BLM2004}). Throughout, we identify $\A$ with its canonical image in $\A^{**}$, so in particular we may view $\A$ as a subalgebra of $\A^{**}$.

Let $\fN$ be a von Neumann algebra with predual $\fN_*$. Let $\iota:\fN_*\to \fN^*$ be the canonical embedding. Let $\psi:\fA\to\fN$ be a bounded linear map.
We let \[\widehat{\psi}=\iota^*\circ \psi^{**}:\fA^{**}\to \fN.\] Then, $\widehat{\psi}$ is a weak-$*$ continuous extension of $\psi$.  In general, $\psi^{**}$ and $\widehat\psi$ do not coincide and it is important to distinguish between them; we will be consistent with our notations to prevent confusion. Note however that $\psi^{**}=\widehat{\psi}$ when $\fN$ is finite-dimensional.

Let $\H$ be a Hilbert space. Multiplication on the left or on the right by a fixed operator is weak-$*$ continuous on $B(\H)$. Therefore, $(B(\H)_*)^\perp\subset B(\H)^{**}$ is a weak-$*$ closed two-sided ideal. Thus, by \cite[Proposition I.10.5]{sakai1971} there is a central projection $\fz\in B(\H)^{**}$ such that
$
\fz B(\H)^{**}=(B(\H)_*)^\perp.
$
In particular, a functional $\phi\in B(\H)^*$ lies in $B(\H)_*$ if and only if
\[
\widehat\phi(\xi)=\widehat\phi((I-\fz)\xi), \quad \xi\in B(\H)^{**}.
\]
Let $\fK$ denote the ideal of compact operators on $\H$.  The following is well known, but we lack a precise reference.

\begin{lemma}\label{L:perpperpcompact}
We have that 
\[
\fK^{\perp\perp}= (I-\fz)B(\H)^{**}.
\]
In particular, if $(a_n)$ is a sequence of compact operators converging to $0$ in the weak-$*$ topology of $B(\H)$, then $( a_n)$ converges to $0$ in the weak-$*$ topology of $B(\H)^{**}$. 
\end{lemma}
\begin{proof}
Let $\phi\in \fK^\perp$. Now, $\fK^{\perp\perp}$ is a weak-$*$ closed two-sided ideal in $B(\H)^{**}$, so that $(I-\fz) \fK\subset \fK^{\perp\perp}$. We conclude that the linear functional
\[
t\mapsto \widehat\phi((I-\fz) t), \quad t\in B(\H)
\]
is weak-$*$ continuous and annihilates $\fK$, and thus is identically zero. This means that $\phi\in  ((I-\fz)B(\H)^{**})_\perp$. Hence,
\[
\fK^{\perp}\subset ((I-\fz)B(\H)^{**})_\perp
\]
or 
\[
(I-\fz)B(\H)^{**}\subset \fK^{\perp\perp}.
\]

Conversely,  given $\phi\in B(\H)^*$ we may use \cite[Theorem 8.4]{paulsen2002} to find a $*$-representation $\pi:B(\H)\to B(\H')$ and unit vectors $v,w\in \H'$ such that
\[
\phi(t)=\langle \pi(t)v,w\rangle, \quad t\in B(\H).
\]
Splitting $\pi$ into a direct sum according to the ideal $\fK$ (see \cite[page 15]{arveson1976inv}), we see that there is $\omega\in  B(\H)_*$ and $\tau\in \fK^\perp$ such that $\phi=\omega+\tau$.  If we assume further than $\phi \in ((1-\fz)B(\H)^{**})_\perp$, then we see for $k\in \fK$ that
\[
\phi(k)=\widehat\phi(\fz k)=\widehat{\omega}(\fz k)+\widehat\tau(\fz k)=\widehat\tau(\fz k)=0
\]
since $\fz \fK\subset \fK^{\perp\perp}$. In other words, we have shown that
\[
 ((I-\fz)B(\H)^{**})_\perp\subset \fK^\perp
\]
which implies
\[
\fK^{\perp\perp}\subset (I-\fz)B(\H)^{**}.
\]
\end{proof}

\subsection{Support projections of representations and non-commutative topology}\label{SS:suppproj}

Let $\fA$ be a $\rC^*$-algebra and let $\pi:\fA\to B(\H_\pi)$ be a $*$-representation. Then, $\pi^{**}:\fA^{**}\to B(\H_\pi)^{**}$ is a weak-$*$ continuous $*$-representation, and so is $\widehat\pi$.  Upon identifying elements with their canonical images inside biduals, we see that $\pi^{**}(t)=\pi(t)$ for every $t\in \fA$. Basic functional analytic arguments reveal that $\ker \pi^{**}=(\ker \pi)^{\perp\perp}$; indeed, both of these objects coincide with $(\ran \pi^*)^\perp$. Because $\fA^{**}$ is a von Neumann algebra and $\ker \pi^{**}$ is a weak-$*$ closed two-sided ideal, there is a central projection $\fe_\pi\in \fA^{**}$ such that
\[
\fe_\pi\fA^{**}=\ker \pi^{**}=(\ker \pi)^{\perp \perp}.
\]
We call the projection $\fs_\pi=I-\fe_\pi$ the \emph{support projection} of $\pi$. It has the property that $\pi^{**}$ is injective on $\fs_\pi \fA^{**}$.
In some simple situations, this projection can be readily identified.

\begin{lemma}\label{L:suppprojmult}
Let $\fA$ be a $\rC^*$-algebra and let $p\in \fA$ be a central projection. Let $\pi:\fA\to \fA$ be the $*$-representation defined as 
\[
\pi(a)=ap, \quad a\in \fA.
\]
Then, $p$ is the support projection of $\pi$.
\end{lemma}
\begin{proof}
A density argument shows that $p$ is also central in $\fA^{**}$. The map 
\[
\xi\mapsto \xi p, \quad \xi\in \fA^{**}
\]
is thus a weak-$*$ continuous $*$-representation agreeing with $\pi$ on $\fA$, so we infer that it coincides with $\pi^{**}$. Thus, $(I-p)\fA^{**}\subset \ker \pi^{**}=\fe_\pi \fA^{**}$ and $I-p\leq \fe_\pi$. Conversely, we find
\[
\fe_\pi p=\pi^{**}(\fe_\pi)=0
\]
so that $\fe_\pi\leq I-p$. This shows that $\fs_\pi=p$.
\end{proof}

We recall here Akemann's so-called \emph{hull-kernel structure}, which can be thought of as a noncommutative topology \cite{akemann1969}.  Let $\fA$ be a $\rC^*$-algebra. A projection $p\in \fA^{**}$ is said to be \emph{open} if there is an increasing net of positive contractions in $\fA$ converging to $p$ in the weak-$*$ topology of $\fA^{**}$. 
A projection $q\in \fA^{**}$ is \emph{closed} if $I-q$ is open in the previous sense. 

Elementary arguments show that $p$ is open if and only if there is a closed left ideal $\fJ\subset \fA$ with $\fJ^{\perp\perp}=\fA^{**}p$ \cite{akemann1969},\cite{akemann1970}.  In particular, if $\pi$ is a $*$-representation of $\fA$, then its support projection $\fs_\pi$ is always closed. We will use this observation implicitly throughout the paper.

Next, we give a well-known two-sided version of a deep theorem of Hay \cite[Theorem 4.2]{hay2007} (see also \cite[Theorem 3.3]{blecher2013} for a streamlined proof).

\begin{lemma}\label{L:Haytwosided}
Let $\fA$ be a unital $\rC^*$-algebra and let $\A\subset \fA$ be a unital subalgebra. Let $q\in \fA^{**}$ be a closed projection lying in $\A^{\perp\perp}$. Then, the two-sided ideal \[(I-q)\A^{**}(I-q)\cap \A\] admits a contractive approximate unit.
\end{lemma}
\begin{proof}
Let $\J=(I-q)\A^{**}\cap \A$. By  \cite[Theorem 3.3]{blecher2013}, $\J$ admits a left contractive approximate unit $(e_i)$ and $\J^{\perp \perp}=(I-q)\A^{**}$.  In turn, by \cite[Proposition 2.5.8]{BLM2004} we see that $(e_i)$ converges to $(I-q)$ in the weak-$*$ topology of $\fA^{**}$. Thus, by \cite[Proposition 2.5.8]{BLM2004} and its proof we obtain that $(I-q)\A^{**}(I-q)\cap \A$ has both a left and a right contractive approximate unit, and hence a contractive approximate unit.
\end{proof}

\subsection{Purity, boundary representations and the $\rC^*$-envelope}\label{SS:bdryrep}

In this paper we are interested in unital operator algebras and their $\rC^*$-envelopes, the latter being defined herein. However, for technical reasons that will become apparent in Section \ref{S:C*lim}, we need to introduce $\rC^*$-envelopes for the more general class of unital operator spaces.

Let $\M\subset B(\H)$ be a unital operator space and let $\psi:\M\to B(\K)$ be a unital completely contractive map. Consider the operator system
\[
\fS=\{a+b^*:a,b\in \M\}\subset B(\H).
\]
We note that $\rC^*(\M)=\rC^*(\fS)$. By \cite[Proposition 3.5]{paulsen2002}, there is a unique completely positive map $\tilde\psi:\fS\to B(\K)$ extending $\psi$. This basic fact is very useful, and shows that unital completely contractive maps on $\M$ correspond, in a natural way, to unital completely positive maps on $\fS$. We use this well-known correspondence tacitly in what follows. In particular, we import many important results from the setting of operator systems and unital completely positive maps to that of unital operator spaces and unital completely contractive maps. 

Given a completely positive map $\phi:\fS\to B(\K)$, we write $\phi\leq  \tilde\psi$ if $\tilde\psi-\phi$ is also completely positive. We say that $\psi$ is \emph{pure} if whenever $\phi\leq \tilde\psi$, there exists a scalar $0\leq \lambda\leq 1$ such that $\phi=\lambda \tilde\psi$. 

If $\psi:\rC^*(\M)\to B(\K)$ is a unital completely positive map, then it is pure if and only if its minimal Stinespring representation $(\sigma_\psi,\K_\psi)$ is irreducible \cite[Corollary 1.4.3]{arveson1969}. Here, $\K_\psi$ is a Hilbert space containing $\K$ and $\sigma_\psi:\rC^*(\M)\to B(\K_\psi)$ is a unital $*$-representation such that
\[
\psi(t)=P_{\K}\sigma_\psi(t)|_{\K}, \quad t\in \rC^*(\M).
\]
See \cite[Theorem 4.1]{paulsen2002} for more details about the existence of the Stinespring representation.

Following \cite{BLM2004}, we say that a \emph{$\rC^*$-extension} of $\M$ is a pair $(\fA,\iota)$ consisting of a unital $\rC^*$-algebra $\fA$ and a unital completely isometric linear map $\iota:\M\to \fA$ such that $\fA=\rC^*(\iota(\M))$. The \emph{$\rC^*$-envelope} of $\M$ is a $\rC^*$-extension $(\fE,u)$ with the property that given any other $\rC^*$-extension $(\fA,\iota)$, there exists a surjective $*$-homomorphism $\pi:\fA\to \fE$ with $\pi\circ \iota=u$. We usually abuse notation slightly and consider just the $\rC^*$-algebra $\fE$ to be the $\rC^*$-envelope, and thus denote it by $\rC^*_e(\M)$; in this case, we identify $\M$ with its image $u(\M)$ in $\fE$. It is readily seen that the $\rC^*$-envelope is essentially unique. The fact that it exists is a non-trivial result of Hamana \cite{hamana1979}. Its universal property implies that the $\rC^*$-envelope is a quotient of every other $\rC^*$-extension through a quotient map that is completely isometric on $\M$. For this reason, we think of the $\rC^*$-envelope as the minimal $\rC^*$-extension. 

Remarkably, when we start with a unital operator algebra $\A$, then the map $u:\A\to \rC^*_e(\A)$ above is actually multiplicative \cite[Proposition 4.3.5]{BLM2004}. In particular, $u(\A)$ is also a unital operator algebra in this case.

To determine the $\rC^*$-envelope, the most commonly used technique is due to Arveson and consists of finding enough $*$-representations of $\rC^*(\M)$ with a certain uniqueness property. Let us be more precise. Let $\pi:\rC^*(\M)\to B(\H_\pi)$ be a unital $*$-representation. We say that it has the \emph{unique extension property} with respect to $\M$ if the only unital completely positive extension of $\pi|_\M$ to $\rC^*(\M)$ is $\pi$ itself. Before proceeding, we record some elementary facts related to the unique extension property.

\begin{lemma}\label{L:directsumUEP}
Let $\M\subset B(\H)$ be a unital operator space. The following statements hold.
\begin{enumerate}[{\rm (i)}]
\item Let $\pi:\rC^*(\M)\to B(\H_\pi)$ be a unital $*$-representation with the unique extension property with respect to $\M$. Let $\X\subset \H_\pi$ be a closed reducing subspace for  $\pi(\rC^*(\M))$. Then, the unital $*$-representation
\[
t\mapsto \pi(t)|_\X, \quad t\in \rC^*(\M)
\]
also has the unique extension property with respect to $\M$.

\item Let
\[
\pi_\lambda:\rC^*(\M)\to B(\H_\lambda), \quad \lambda\in \Lambda
\]
be a set of unital $*$-representations with the unique extension property with respect to $\M$. Then, the unital $*$-representation 
\[
t\mapsto \bigoplus_{\lambda\in \Lambda}\pi_\lambda(t), \quad t\in \rC^*(\M)
\]
also has the unique extension property with respect to $\M$.
\end{enumerate}
\end{lemma}
\begin{proof}
(i) Let $\psi:\rC^*(\M)\to B(\X)$ be a unital completely positive map that satisfies
\[
\psi(a)=\pi(a)|_\X, \quad a\in \M.
\]
Define $\Psi:\rC^*(\M)\to B(\H_\pi)$ as
\[
\Psi(t)=\psi(t)\oplus \left(\pi(t)|_{\X^\perp}\right), \quad t\in \rC^*(\M).
\]
Then, we see that $\Psi$ is a unital completely positive map that agrees with $\pi$ on $\M$. By assumption, this means that $\Psi=\pi$ and in particular 
\[
\psi(t)=\pi(t)|_\X, \quad t\in \rC^*(\M).
\]

(ii) This is  \cite[Proposition 4.4]{arveson2011}.
\end{proof}

Let $\pi:\rC^*(\M)\to B(\H_\pi)$ be a $*$-representation. Consider its essential subspace
\[
\E_\pi=\ol{\pi(\rC^*(\M))\H_\pi}.
\]
Then, $\pi$ is said to have the \emph{essential unique extension property} with respect to $\M$ if the unital $*$-representation $\rho:\rC^*(\M)\to B(\E_\pi)$ defined as 
\[
\rho(t)=\pi(t)|_{\E_\pi}, \quad t\in \rC^*(\M)
\]
has the usual unique extension property with respect to $\M$.

\begin{lemma}\label{L:UEPisom}
Let $\M\subset B(\H)$ be a unital operator space and let $\pi:\rC^*(\M)\to B(\H_\pi)$ be a unital $*$-representation with the unique extension property with respect to $\M$. Let $\K$ be another Hilbert space and let $V:\H_\pi\to \K$ be an isometry. Define $\pi':\rC^*(\M)\to B(\K)$ as
\[
\pi'(t)= V\pi(t)V^*, \quad t\in \rC^*(\M).
\]
Then, $\pi'$ is a $*$-representation with the essential unique extension property with respect to $\M$. 
\end{lemma}
\begin{proof}
Since $V$ is an isometry, it is readily verified that $\pi'$ is a $*$-representation.  Let $U=P_{V\H_\pi}V:\H_\pi\to V\H_\pi$. Then, $U$ is unitary and with respect to the decomposition $\K=V\H_\pi\oplus (V\H_\pi)^\perp$ we find
\[
\pi'(t)=U\pi(t)U^* \oplus 0, \quad t\in \rC^*(\M).
\]
It is easily verified that $U\pi(\cdot)U^*$ is a unital $*$-representation of $\rC^*(\M)$ with the unique extension property with respect to $\M$, which immediately implies the desired statement.
\end{proof}

Due to important contributions of Muhly--Solel \cite{MS1998} and Dritschel--McCullough \cite{dritschel2005}, it is known that unital $*$-representations of $\rC^*(\M)$ with the unique extension property with respect to $\M$ are always plentiful. In fact, more is true. As the next result shows, we can also require these $*$-representations to be irreducible. 

If an irreducible $*$-representation $\pi$ of $\rC^*(\M)$ has the unique extension with respect to $\M$, then we say that $\pi$ is a \emph{boundary representation} for $\M$ on $\rC^*(\M)$.  Recall that a unital completely contractive map $\psi:\M\to B(\F)$ is called a \emph{matrix state} whenever $\F$ is finite-dimensional. The following improves on \cite[Theorem 3.1]{kleski2014bdry} slightly.

\begin{theorem}\label{T:bdryexist}
Let $\M\subset B(\H)$ be a unital operator space, let $n\in \bN$ and let $A\in \bM_n(\M)$. Then, there is a boundary representation $\beta:\rC^*(\M)\to B(\H_\beta)$ for $\M$ such that 
\[
\|A\|=\|\beta^{(n)}(A)\|.
\]
Furthermore, there is a finite-dimensional subspace $\F\subset \H_\beta$ such that if we define a matrix state $\psi:\M\to B(\F)$ as
\[
\psi(b)=P_\F\beta(b)|_\F, \quad b\in \M
\]
then we have
\[
\|A\|=\|\psi^{(n)}(A)\|.
\]
\end{theorem}
\begin{proof}
By  the proof of \cite[Theorem 2.5]{kleski2014bdry}, there is a two-dimensional Hilbert space $\E$ and a pure matrix state $\phi:\bM_n(\M)\to B(\E)$ such that $\|\phi(A)\|=\|A\|$. 
In turn, by \cite[Theorem 2.4]{davidsonkennedy2015} we find a boundary representation $\gamma:\rC^*(\bM_n(\M))\to B(\H_\gamma)$ for $\bM_n(\M)$ with the property that
\[
\phi(B)=P_{\E}\gamma(B)|_\E, \quad B\in \bM_n(\M).
\]
Note now that $\rC^*(\bM_n(\M))=\bM_n(\rC^*(\M))$. It follows from the last paragraph in \cite[page 485]{hopenwasser1973} that there is an irreducible $*$-representation $\beta:\rC^*(\M)\to B(\H_\beta)$ and a unitary operator $U:\H_\beta^{(n)}\to \H_\gamma$ with the property that
\[
\gamma(T)=U \beta^{(n)}(T)U^*, \quad T\in \bM_n(\rC^*(\M)).
\]
By the theorem on \cite[page 486]{hopenwasser1973}, we conclude that $\beta$ is a boundary representation for $\M$. Moreover,
\[
\|\beta^{(n)}(A)\|=\|\gamma(A)\|\geq \|\phi(A)\|=\|A\|
\]
so in fact $\|\beta^{(n)}(A)\|=\|A\|$. 

Next, we see that
\begin{align*}
P_{U^*\E}\beta^{(n)}(A)|_{U^*\E}&=U^* P_\E U\beta^{(n)}(A)|_{U^*\E}\\
&=U^* P_\E \gamma(A)U|_{U^*\E}\\
&=V^* (P_\E \gamma(A)|_{\E})V
\end{align*}
where $V=P_\E U|_{U^*\E}:U^*\E\to \E$ is a unitary operator. Thus, we conclude that 
\[
\|P_{U^*\E}\beta^{(n)}(A)|_{U^*\E}\|=\|\phi(A)\|=\|A\|.
\]
Finally, choose any finite-dimensional subspace $\F\subset \H_\beta$ with the property that $\F^{(n)}$ contains the  two-dimensional subspace $U^*\E$ and define $\psi:\M\to B(\F)$ as
\[
\psi(b)=P_\F\beta(b)|_\F, \quad b\in \M.
\]
Then, we find
\[
\|\psi^{(n)}(A)\|=\|P_{\F^{(n)}} \beta^{(n)}(A)|_{\F^{(n)}}\|\geq \|P_{U^*\E}\beta^{(n)}(A)|_{U^*\E}\|=\|A\|
\]
as desired.
\end{proof}

By Theorem \ref{T:bdryexist}, there always exists a set $\Delta$ of boundary representations for $\M$ on $\rC^*(\M)$ with the property that $\bigoplus_{\beta\in \Delta}\beta$ is completely isometric on $\M$ (see also \cite[Theorem 7.1]{arveson2008} and  \cite[Theorem 3.3]{davidsonkennedy2015}). Let $\Sigma_\Delta\subset \rC^*(\M)$ denote the kernel of $\bigoplus_{\beta\in \Delta}\beta$. We may invoke \cite[Theorem 2.2.3]{arveson1969} to see  that $\Sigma_\Delta$ coincides with the intersection of the kernels of all boundary representations for $\M$ on $\rC^*(\M)$, and is thus independent of $\Delta$. This ideal is called the \emph{Shilov ideal} of $\M$ in $\rC^*(\M)$.  By \cite[Proposition 2.2.4 and Theorem 2.2.5]{arveson1969} we conclude that $\rC^*_e(\M)$ is $*$-isomorphic to $\rC^*(\M)/\Sigma_\Delta$. For this reason, we usually think of the collection of boundary representations of $\M$ as the \emph{non-commutative Choquet boundary} of $\M$. 
We also note here that the Shilov ideal of $\M$ in $\rC^*_e(\M)$ is trivial \cite[Proposition 2.2.4]{arveson1969}.

If $\beta$ is a boundary representation for $\M$ on $\rC^*(\M)$, then $\ker \beta$ contains the Shilov ideal by \cite[Proposition 2.2.3]{arveson1969}. Therefore, there is an irreducible $*$-representation $\tilde\beta$ on $\rC^*_e(\M)$ such that $\beta=\tilde\beta\circ q$, where $q$ is the quotient map on $\rC^*(\M)$ corresponding to the Shilov ideal. In other words, a boundary representation always ``factors through the $\rC^*$-envelope".

 The next observation is simple but useful.

\begin{lemma}\label{L:bdrysepequiv}
Let $\M$ be a unital operator space and let $\Delta$ be a set of irreducible $*$-representations of $\rC^*_e(\M)$. Then, $\Delta$ is completely norming for $\M$ if and only if it is separating for $\rC^*_e(\M)$. In particular, there is a set of boundary representations for $\M$ on $\rC^*_e(\M)$ which is separating for $\rC^*_e(\M)$.
\end{lemma}
\begin{proof}
Let $\sigma=\bigoplus_{\pi\in \Delta}\pi$. If $\Delta$ is completely norming for $\M$, then $\ker \sigma$ is the Shilov ideal of $\M$ in $\rC^*_e(\M)$, which is trivial as mentioned above. Thus, $\sigma$ is injective and hence $\Delta$ is separating for $\rC^*_e(\M)$. Conversely, if $\Delta$ is separating for $\rC^*_e(\M)$ then $\sigma$ is injective and hence completely isometric. The second statement follows from the first and from Theorem \ref{T:bdryexist}.
\end{proof}

In the classical situation of a unital subspace $\M$ of continuous functions on a compact Hausdorff space, if $\M$ separates the points then it is known that the Choquet boundary is always dense in the Shilov boundary \cite[Proposition 6.4]{phelps2001}. There are several reasonable candidates for a non-commutative analogue of the Shilov boundary, and for all of them we have a similar density property.

\begin{theorem}\label{T:density}
Let $\M$ be a unital operator space and let $\B\subset \widehat{\rC^*_e(\M)}$ denote the set of unitary equivalence classes of boundary representations for $\M$. Then, the following statements hold.
\begin{enumerate}[{\rm (i)}]
\item Let $\S_\B(\rC^*_e(\M))$ denote the set of states on $\rC^*_e(\M)$ whose GNS representation is a boundary representation for $\M$. Then, $\S_\B(\rC^*_e(\M))$ is weak-$*$ dense in the set of pure states on $\rC^*_e(\M)$.
\item The set
$
\{\ker \beta:[\beta]\in \B\}
$
is dense in $\Prim (\rC^*_e(\M))$.

\item The set $\B$ is dense in  $\widehat{\rC^*_e(\M)}$.
\end{enumerate}
\end{theorem}
\begin{proof}
Statement (ii) follows from the definition of the Jacobson topology on the primitive ideal space along with Lemma \ref{L:bdrysepequiv}. Furthermore, by definition of the topology on $\widehat{\rC^*_e(\M)}$, (ii) immediately implies (iii). Finally, (i) is always equivalent to (iii) by Lemma \ref{L:densestaterep}.
\end{proof}


\section{Non-commutative peak points}\label{S:ncpeakpoints}

The goal of this section is to introduce an important tool to detect ``isolated" irreducible $*$-representations, and ultimately to exhibit finite-dimensional boundary representations. The rough idea is to view certain irreducible $*$-representations as analogues of classical peak points. Let us be more precise.

 Let $\fA$ be a $\rC^*$-algebra and let  $\pi:\fA\to B(\H_\pi)$ be an  irreducible $*$-representation. Then, $\pi$ is said to be a \emph{peaking representation for $\fA$} if there is $n\in \bN$ and $T\in \bM_n(\fA)$ with the property that for every irreducible $*$-representation $\sigma:\fA\to B(\H_\sigma)$ not unitarily equivalent to $\pi$ we have
 \[
\|\pi^{(n)}(T)\|> \|\sigma^{(n)}(T)\|.
\]
In this case, we say that $T$ \emph{peaks} at $\pi$.   For our purposes, we will need a slight weakening of this notion, which in a sense is ``local". We say that $T$ \emph{peaks locally at $\pi$} if for every irreducible $*$-representation $\sigma:\fA\to B(\H_\sigma)$ not unitarily equivalent to $\pi$ and every finite-dimensional subspace $\F\subset \H_\sigma$  we have
\[
\|\pi^{(n)}(T)\|> \|P_{\F^{(n)}}\sigma^{(n)}(T)|_{\F^{(n)}}\|.
\]
 If $\S\subset \fA$ is a subset and there is $n\in \bN$ along with $T\in \bM_n(\S)$ such that $T$ peaks locally at $\pi$, then $\pi$ is said to be a \emph{locally peaking representation for $\S$}. 
 
 As a consequence of Theorem \ref{T:peakprojlocpeak} below, we will also exhibit an example of a natural $\rC^*$-algebra for which there is a locally peaking representation but there are no peaking representations (see Example \ref{E:locpeak}). Hence, our local notion is genuinely more flexible than the original one.

The reader may wonder why unitary equivalence is used instead of weak equivalence in the previous definitions; this alternative would also encode some kind of peaking phenomenon. There are three reasons motivating our choice. First, what we are hoping to accomplish using locally peaking representations is to identify boundary representations (see Theorem \ref{T:peakingrep}). Using weak equivalence rather than unitary equivalence in our definition would result in a class of representations which does not necessarily consist entirely of boundary representations.  Indeed, this can be inferred from  \cite[Example 6.6.3]{DK2019}; the authors are grateful to Matthew Kennedy for pointing this out. Second, guided by Question \ref{Q:main}, in this paper our focus is on finite-dimensional representations, in which case unitary equivalence coincides with weak equivalence anyway, by Lemma \ref{L:fdimweakcontain}. Therefore, the proposed alternative version of (locally) peaking projections would reduce to the original one in our case of interest. The third reason is one of consistency: as mentioned above, previous occurrences of peaking representations have used unitary equivalence \cite{arveson2011},\cite{NPSV2018},\cite{CDHLZ2019},\cite{DP2020} so we opt here to follow this trend. 

We now arrive at an important result, which illustrates  why the notion of peaking representation is relevant for Question \ref{Q:main}.  The reader should compare it with \cite[Theorem 7.2]{arveson2011} (see also \cite[Theorem 5.2]{NPSV2018} for a related result).

\begin{theorem}\label{T:peakingrep}
Let $\A\subset B(\H)$ be a unital operator algebra  and let $\pi$ be an  irreducible $*$-representation of $\rC^*(\A)$. Assume that $\pi$ is a locally peaking representation for $\A$. Then, $\pi$ is a boundary representation for $\A$.
\end{theorem}
\begin{proof}
There is $n\in \bN$ along with $A\in \bM_n(\A)$ such that
\[
\|A\|\geq \|\pi^{(n)}(A)\|> \|P_{\F^{(n)}}\sigma^{(n)}(A)|_{\F^{(n)}}\|
\]
for every irreducible $*$-representation $\sigma:\rC^*(\A)\to B(\H_\sigma)$ not unitarily equivalent to $\pi$ and every finite-dimensional subspace $\F\subset \H_\sigma$. 
In view of Theorem \ref{T:bdryexist}, there is a boundary representation $\beta:\rC^*(\A)\to B(\H_\beta)$ for $\A$ and a finite-dimensional subspace $\G\subset \H_\beta$ such that 
\[
\|P_{\G^{(n)}}\beta^{(n)}(A)|_{\G^{(n)}}\|=\|A\|. 
\]
Necessarily, this implies that $\pi$ is unitarily equivalent to $\beta$, and hence is a boundary representation.
\end{proof}

In light of this result, it is desirable to efficiently identify locally peaking representations. This will be accomplished in the following subsection with the aid of certain special projections.

\subsection{Peaking projections}\label{SS:peakproj}
As part of our overarching analogy with classical peak points, the previous developments replaced points with irreducible $*$-representations to arrive at the notion of peaking representations. Another non-commu\-tative interpretation of peak points replaces points with projections in the bidual; these should be thought of as analogues of characteristic functions.

In this subsection, we show how certain peaking projections can be used to produce locally peaking representations. This is beneficial for us, as peaking projections have been thoroughly studied since the appearance of \cite{hay2007} (see also \cite{blecher2013} and the references therein). In particular, there are known deep results that characterize these projections via concrete conditions that we will manage to verify in some cases.

We recall the definition (see \cite[Definition 3.5 and Theorem 5.1]{hay2007}).
A closed projection $q\in \fA^{**}$ is a \emph{peaking projection} if there is a contraction $t\in \fA$ with the property that $tq=q$ and $\|tp\|<1$ whenever $p\in \fA^{**}$ is a closed projection with $pq=0$. In this case, we say that $t$ \emph{peaks} at $q$. Note then that necessarily $\|t\|=1$. In addition, if there is a $*$-representation $\pi$ of $\fA$ such that $\fs_\pi=q$, then 
\[
\pi(t)=\pi^{**}(t\fs_\pi )=\pi^{**}(\fs_\pi )=I.
\]
We aim to give a refinement of a deep non-commutative counterpart to the classical Glicksberg peak point theorem \cite[Theorem II.12.7]{gamelin1969} based on non-trivial work of Hay \cite{hay2007}, Blecher--Hay--Neal \cite{BHN2008} and Read \cite{read2011}. The sharper conclusion that we obtain upon specializing to the separable setting does not seem to have been previously recorded, although it may have been known to experts. The fact that separability allows for a cleaner result is a reflection of the familiar fact that the theory of peak points for uniform algebras is much more definitive in the case where the underlying compact Hausdorff space is metrizable; see for instance \cite[Lemma II.12.2]{gamelin1969}.

\begin{theorem}\label{T:peakOA}
Let $\A\subset B(\H)$ be a separable unital operator algebra and let $q\in \A^{\perp\perp}$ be a closed projection. Then, there is an element $t\in \A$ peaking at $q$ such that $1/4 I\leq t^*t\leq I$. Furthermore, $t$ has the property that $\|\widehat\psi(q)\|=1$ whenever $\psi$ is a matrix state on $\rC^*(\A)$ satisfying $\|\psi(t^*t)\|=1$. 
\end{theorem}
\begin{proof}
We may assume that $q\notin \A$, for otherwise we may simply choose $t=q$. 
By Lemma \ref{L:Haytwosided}, the algebra $(I-q)\A^{**} (I-q)\cap \A$ has a contractive approximate unit. By \cite[Theorem 1.1]{read2011}, there is in fact a contractive approximate unit $(b_j)_{j\in J}$ such that $\|I-b_j\|\leq 1$ for every $j\in J$.  Observe that $ b_j(I-q) =b_j$ or $b_jq =0$ for every $j\in J$. Next, because $\A$ is norm separable, so is the subset $\{b_j:j\in J\}$. Let $\Gamma\subset \{b_j:j\in J\}$ be a countable dense subset. Note that $q\notin \A$ and $(b_j)$ converges to $I-q$ in the weak-$*$ topology of $B(\H)^{**}$ \cite[Proposition 2.5.8]{BLM2004}, so that $\Gamma$ is infinite and we write $\Gamma=\{c_n:n\in \bN\}$. Since $(I-c_n)^*(I-c_n)\leq I$, we infer $\re c_n\geq 0$ for every $n\geq 1$. Define 
\[t=\frac{1}{2}I+\sum_{n=1}^\infty \frac{1}{2^{n+1}}(I-c_n)\in \A.\] 
We conclude that $\|t\|\leq 1$ and $tq=q$. Note that 
\begin{align*}
t^*t=\frac{1}{4}I+\sum_{n=1}^\infty \frac{1}{2^{n+1}}(I-\re c_n)+\left( \sum_{n=1}^\infty \frac{1}{2^{n+1}}(I-c_n)\right)^*\left(\sum_{n=1}^\infty \frac{1}{2^{n+1}}(I-c_n)\right).
\end{align*}
Since $c_n$ is a contraction, we see that $I-\re c_n\geq 0$ for every $n\geq 1$, whence $t^*t\geq 1/4 I$.
Moreover, if $\psi$ is a matrix state on $\rC^*(\A)$ then 
\[
0\leq I-\psi(\re c_n)\leq I, \quad n\geq 1
\]
whence
\begin{align*}
\|\psi(t^*t)\|&\leq \frac{1}{4}+\sum_{n=1}^\infty \frac{1}{2^{n+1}}\|I-\psi(\re c_n)\|+\left(\sum_{n=1}^\infty \frac{1}{2^{n+1}}\right)^2\\
&\leq \frac{1}{4}+\sum_{n=1}^\infty \frac{1}{2^{n+1}}+\frac{1}{4}=1.
\end{align*}
Hence $\|\psi(t^*t)\|=1$ holds only when $\|1- \psi(\re c_n)\|=1$ for every $n\geq 1$. Because $(b_j)$ converges to $I-q$ in the weak-$*$ topology of $B(\H)^{**}$, there is a subsequence of $(\psi(c_n))$ that converges to $\widehat{\psi}(I-q)\geq 0$  in \emph{norm}, seeing as $\psi$ is a matrix state. Consequently, a subsequence of $(\psi(\re c_n))$ converges to $\widehat{\psi}(I-q)$ in norm, and
\[
\|\widehat{\psi}(q)\|\geq \liminf_{n\to\infty} \|I-\psi(\re c_n)\|.
\] 
We conclude that $\|\widehat{\psi}(q)\|=1$ whenever $\|\psi(t^*t)\|=1$. Invoking  \cite[Theorem 5.1]{hay2007} we see that $t$ peaks at $q$.
\end{proof}

The previous result allows us to give characterizations, in the separable setting, of peaking projections.

\begin{corollary}\label{C:peakOAchar}
Let $\A\subset B(\H)$ be a separable unital operator algebra and let $q\in B(\H)^{**}$ be a closed projection. Then, there is a contraction $t\in \A$ that peaks at $q$ if and only if $q\in \A^{\perp\perp}$.
\end{corollary}
\begin{proof}
This follows upon combining \cite[Proposition 5.6]{hay2007} with Theorem \ref{T:peakOA}.
\end{proof}

\subsection{Peaking support projections}\label{SS:peakprojbdryi}
We now return to our analysis of peaking representations. Our first task is to show that irreducible $*$-representations with a peaking support projection must be locally peaking.
This is consistent with the classical commutative setting. Indeed, let $X$ be a compact metric space and let $\A\subset C(X)$ be a uniform algebra. For $\xi\in X$, we let $\pi_\xi:C(X)\to \bC$ be the character of evaluation at $\xi$. It is well known that $\fs_{\pi_\xi}\fs_{\pi_{\xi'}}=0$ whenever $\xi\neq \xi'$ (see for instance \cite[Proposition 2.3]{clouatre2018lochyp}). In particular, we see that if some contraction $a\in \A$ peaks at $\fs_{\pi_\xi}$, then $\xi$ is a peak point for the function $a$, which is equivalent to $a$ peaking at $\pi_\xi$. A similar statement holds in the general non-commutative context, at least locally.

\begin{theorem}\label{T:peakprojlocpeak}
Let $\A\subset B(\H)$ be a separable unital operator algebra and let $\pi$ be a finite-dimensional irreducible $*$-representation of $\rC^*(\A)$ such that $\fs_\pi\in \A^{\perp\perp}$. Then, $\pi$ is a locally peaking representation for $\A$.
\end{theorem}
\begin{proof}
Let $a\in \A$ be the contraction obtained from an application of Theorem \ref{T:peakOA} to the closed projection $\fs_\pi\in \A^{\perp\perp}$. In particular, $a\in \A$ peaks at $\fs_\pi$ so that
\[
1=\|a\|=\|\pi(a)\|.
\]
Let $\sigma:\rC^*(\A)\to B(\H_\sigma)$ be an irreducible $*$-representation  and let $\F\subset \H_\sigma$ be a finite-dimensional subspace. Define a matrix state $\psi:\rC^*(\A)\to B(\F)$ as
\[
\psi(t)=P_\F \sigma(t)|_\F, \quad t\in \rC^*(\A).
\]
Assume that $\|\psi(a)\|=1$. By the Schwarz inequality \cite[Proposition 3.3]{paulsen2002}, we infer that $\|\psi(a^*a)\|=1$, whence $\|\widehat\psi(\fs_\pi)\|=1$ by choice of $a$. Now, it is readily verified that
\[
\widehat\psi(\xi)=P_\F \widehat\sigma(\xi)|_\F, \quad \xi\in \rC^*(\A)^{**}.
\]
In particular, we infer that $\| \widehat\sigma(\fs_\pi)\|=1$. Since $\sigma$ is assumed to be irreducible, so is $\widehat\sigma$. In turn, since $\fs_\pi$ is central we must have that $\widehat\sigma(\fs_\pi)=I$. This implies that $\ker \pi^{**}\subset \ker \widehat\sigma$, and consequently $\ker \pi\subset \ker \sigma$. In other words, $\sigma$ is weakly contained in $\pi$. By Lemma \ref{L:fdimweakcontain} we conclude that $\pi$ is unitarily equivalent to $\sigma$. This shows that $a$ peaks locally at $\pi$.
\end{proof}

We now wish to give a non-trivial example where the previous result applies.

\begin{example}\label{E:peaking}
Let $\H=\oplus_{m=1}^\infty \bC^m$ and let $\fA=\prod_{m=1}^\infty \bM_{m}$. For each $m\in \bN$, we let $\pi_m:\fA\to \bM_m$ denote the coordinate projection.
Let $\fK\subset \fA$ denote the ideal of compact operators in $\fA$, so that $\fK=\oplus_{m=1}^\infty \bM_m$.
Fix $N\in \bN$ and let 
\[
\chi_N=(0,0,\ldots,I_N,0,\ldots)\in \fK.
\]
Let $(K_r)$ be a sequence in $\fK$ converging to $0$ in the weak-$*$ topology of $B(\H)$. 
Let $\A\subset \fA$ be a separable unital subalgebra containing  $\chi_N+K_r$ for every $r\in \bN$.  It follows from Lemma \ref{L:perpperpcompact} that $(\chi_N+K_r)$ converges to $\chi_N$ in the weak-$*$ topology of $B(\H)^{**}$, whence $\chi_N\in \A^{\perp\perp}$.

Let $\rho=\pi_N|_{\rC^*(\A)}:\rC^*(\A)\to \bM_N$.  We claim that $\chi_N=\fs_{\rho}$. First, we note that $\rho$ is finite-dimensional so that $\rho^{**}=\widehat{\rho}$. 
We find
\[
\rho^{**}(\chi_N)=\lim_{r\to\infty}\rho(\chi_N+K_r)=I
\]
whence $\fs_{\rho}\leq\chi_N$. Hence, $\fs_\rho\in \chi_N \rC^*(\A)^{**}$. 
Since $\pi_N|_{\chi_N \fA}$ is isometric, so is $(\pi^{**}_N)|_{\chi_N \fA^{**}}$. Because 
\[
\pi^{**}_N(\chi_N)=\widehat{\pi_N}(\chi_N)=\pi_N(\chi_N)=I=\widehat\rho(\fs_\rho)=\rho^{**}(\fs_\rho)=\pi^{**}_N(\fs_\rho)
\]
we must have $\fs_\rho=\chi_N$, as claimed. 

Finally, if $\A$ is chosen to be rich enough so that $\pi_N(\rC^*(\A))=\bM_N$, then $\rho$ is irreducible, and Theorem \ref{T:peakprojlocpeak} implies that $\rho$ is a locally peaking representation for $\A$. \qed
\end{example}

As another application of Theorem \ref{T:peakprojlocpeak}, we can also show how the notion of a locally peaking representation is genuinely weaker than that of a peaking representation.

\begin{example}\label{E:locpeak}
Let $H^2$ denote the classical Hardy space on the open unit disc $\bD\subset \bC$. This is the Hilbert space of holomorphic functions $f:\bD\to \bC$ that can be written as
\[
f(w)=\sum_{n=0}^\infty a_n w^n, \quad w\in \bD
\]
with \[\sum_{n=0}^\infty |a_n|^2<\infty.\] Let $A(\bD)$ denote the disc algebra, which is the closed subalgebra of $C(\ol{\bD})$ consisting of those functions that are holomorphic on $\bD$. In this example, we will freely use many classical facts about these objects. The reader is referred to \cite{hoffman1988},\cite{douglas1998},\cite{agler2002} for greater detail.

Given $\phi\in A(\bD)$, the corresponding multiplication operator $M_\phi$ on $B(H^2)$ is known to be bounded with $\|M_\phi\|=\|\phi\|$. In fact, there is a unital completely isometric homomorphism $\Theta:A(\bD)\to B(H^2)$ defined as
\[
\Theta(\phi)=M_\phi, \quad \phi\in A(\bD).
\]
Let $\fT=\rC^*(\Theta(A(\bD)))$. Then, this $\rC^*$-algebra contains the ideal of compact operators on $H^2$ and the corresponding quotient $\fT/\fK$ can be naturally identified $*$-isomorphically with $C(\bT)$ via a map sending $M_z+\fK$ to $z$ (here, by $z$ we denote the identity function). Using this identification, for each $\zeta\in \bT$ we let $\chi_\zeta:\fT\to \bC$ be the character defined as
\[
\chi_\zeta(t)=(t+\fK)(\zeta), \quad t\in \fT.
\]
Basic representation theory of $\rC^*$-algebras implies that up to unitary equivalence, the irreducible $*$-representations of $\fT$ are precisely the identity representation and the characters $\chi_{\zeta}$ for $\zeta\in \bT$. 

It is readily seen that $\fT$ has simply no peaking representation for $\Theta(A(\bD))$ (see the discussion following \cite[Remark 3.4]{kleski2014bdry}). We now show on the other hand that  $\chi_1$ is a locally peaking representation for $\Theta(A(\bD))$. To see this, define the contraction
\[
a=\frac{1}{2}(I+M_z)\in \Theta(A(\bD)).
\]
We claim that $a$ peaks on $\fs_{\chi_1}$. First note that $\chi_1(a)=1$ so that $I-a\in \ker \chi_1$. In particular, $(I-a)\fs_{\chi_1}=0$ so that $a\fs_{\chi_1}=\fs_{\chi_1}$. To show that $a$ peaks on $\fs_{\chi_1}$, by virtue of \cite[Theorem 5.1]{hay2007} it suffices to fix a pure state $\psi$ on $\fT$ with $\psi\neq \chi_1$ and show that $\psi(a^*a)<1$. In view of the above description of the irreducible $*$-representations of $\fT$, we infer that either $\psi=\chi_\zeta$ for some $\zeta\in \bT\setminus\{1\}$ or there is a unit vector $h\in H^2$ such that
\[
\psi(t)=\langle th,h \rangle, \quad t\in \fT.
\] 
Assume first that $\psi=\chi_\zeta$ for some $\zeta\in \bT\setminus\{1\}$. It is readily verified that 
\[
\psi(a^*a)=\frac{ |1+\zeta|^2}{4}<1.
\]
Next, let $h\in H^2$ be a unit vector and assume that
\[
\psi(t)=\langle th,h \rangle, \quad t\in \fT.
\]
If $\psi(a^*a)=1$, then the Cauchy-Schwarz inequality implies that $a^*a h=h$. By definition of $a$, using that $M_z$ is an isometry we infer that 
\[
\frac{1}{2}(M_z+M_z^*)h=h.
\]
Since $M_z h=zh$ and $M_z^*h=(h-h(0))/z$, the previous equation is equivalent to 
\[
(w-1)^2 h(w)=h(0), \quad w\in\bD.
\]
An elementary calculation reveals that the condition $h\in H^2$ forces $h=0$, which is absurd. Hence $\psi(a^*a)<1$. We conclude that $a$ indeed peaks at $\fs_{\chi_1}$. By Corollary \ref{C:peakOAchar} and Theorem \ref{T:peakprojlocpeak}, we conclude that $\chi_1$ is a locally peaking representation for $\A$. 
\qed
\end{example}

\subsection{Strongly peaking representations}\label{SS:unifpeakingrep}
We now turn to a uniform version of the notion of peaking representation. 
Let $\fA$ be a $\rC^*$-algebra and let  $\pi$ be an  irreducible $*$-representation of $\fA$. We let
\[
\U_\pi=\{[\sigma]\in \widehat{\fA}: [\sigma]\neq [\pi]\}.
\]
 Let $n\in \bN$ and $T\in \bM_n(\fA)$. Then, $T$ is said to \emph{peak strongly at $\pi$} if
\[
\|\pi^{(n)}(T)\|>\sup_{[\sigma]\in \U_\pi} \|\sigma^{(n)}(T)\|.
\]
If $\S\subset \fA$ is a subset and there is $n\in \bN$ along with $T\in \bM_n(\S)$ that peaks strongly at $\pi$, then $\pi$ is said to be a \emph{strongly peaking representation for $\S$}. 

We remark here that while strongly peaking representations are a major theme of the recent paper \cite{DP2020}, there is seemingly very little overlap between that paper and our current work. We refer the reader to \cite{DP2020} for connections between these representations and certain uniqueness properties of so-called ``fully compressed" operator systems.

If $\A$ is a uniform algebra on some compact metric space $X$, then the peaking representations for $\A$ are precisely the characters of evaluation at points in the Choquet boundary  \cite[Theorem 11.3]{gamelin1969}, so in particular they exist in abundance. On the other hand, the continuity of the functions in $\rC^*(\A)$ implies that strongly peaking representations for $\rC^*(\A)$ cannot exist unless $X$ has isolated points. We conclude from this that strongly peaking representations may be quite rare. When they do exist however, then enjoy many useful properties, as we shall see below. We start with an observation that can be extracted from the proof of \cite[Corollary 2.6]{DP2020}.

\begin{lemma}\label{L:unifpeakC*env}
Let $\A$ be a unital operator algebra and let $\pi$ be an irreducible $*$-representation of $\rC^*_e(\A)$. Then, $\pi$ is a strongly peaking representation for $\A$ if and only if it is a strongly peaking representation for $\rC^*_e(\A)$.
\end{lemma}
\begin{proof}
Assume that $\pi$ is a strongly peaking representation for $\rC^*_e(\A)$. Thus, there is $n\in \bN$ and $T\in \bM_n(\rC^*_e(\A))$ with the property that
\[
\|\pi^{(n)}(T)\|>\sup_{[\sigma]\in \U_\pi}\|\sigma^{(n)}(T)\|.
\]
For each $u\in \U_\pi$, let $\sigma_u$ be an irreducible $*$-representation of $\rC^*_e(\A)$ such that $[\sigma_u]=u$. 
Then, the $*$-representation $\bigoplus_{u\in \U_\pi}\sigma_u$ is not completely isometric on $\rC^*_e(\A)$ and hence it is not injective either. By Lemma \ref{L:bdrysepequiv}, we conclude that the set $\{\sigma_u:u\in \U_\pi\}$ is not completely norming for $\A$, so there is $m\in \bN$ and $A\in \bM_m(\A)$ such that
\[
\sup_{[\sigma]\in \U_\pi}\|\sigma^{(m)}(A)\|<\|A\|.
\]
On the other hand, the set $\{\pi\}\cup \{\sigma_u:u\in \U_\pi\}$ is clearly completely norming for $\rC^*_e(\A)$, so we must have that $\|A\|=\|\pi^{(m)}(A)\|$. We conclude that $A$ peaks strongly at $\pi$. The converse is trivial.
\end{proof}

 Recall that in the setting of Question \ref{Q:main}, we are dealing with RFD $\rC^*$-algebras. Interestingly, strongly peaking representations must necessarily be finite-dimensional in that case. In fact, we can show even more.
 
\begin{theorem}\label{T:peakingRFD}
Let $\fA$ be an  RFD $\rC^*$-algebra and let $\pi$ be an irreducible $*$-repre\-sentation. Assume that there is $n\in \bN$ and $T\in \bM_n(\fA)$ with the property that
\[
\|\pi^{(n)}(T)\|>\sup_{[\sigma]\in \W_\pi} \|\sigma^{(n)}(T)\|
\]
where $\W_\pi=\{[\sigma]\in \widehat\fA:\sigma\not\sim\pi\}$. Then, $\pi$ is finite-dimensional.
\end{theorem}
\begin{proof}
Write $\pi:\fA\to B(\H_\pi)$ and choose  $\eps>0$ with
\[
\|\pi^{(n)}(T)\|\geq (1+\eps)\|\sigma^{(n)}(T)\|, \quad \sigma\in \W_\pi.
\]
By virtue of Theorem \ref{T:EL}, there is a net of $*$-representations
\[
\pi_\lambda:\fA\to B(\H_\pi), \quad \lambda\in \Lambda
\]
such that $\pi_\lambda(\fA)$ is finite-dimensional for every $\lambda\in \Lambda$ and
\[
\lim_{\lambda} \pi_\lambda(s)\xi=\pi(s)\xi, \quad s\in \fA, \xi\in \H_\pi.
\]
Choose $0<\delta<1$ such that $\delta(1+\eps)>1$. There is $\mu\in \Lambda$ such that
\[
\|\pi_\mu^{(n)}(T)\|\geq \delta \|\pi^{(n)}(T)\|.
\]
Since $\pi_\mu(\fA)$ is finite-dimensional, there are finite-dimensional irreducible $*$-represen\-tations $\rho_1,\ldots,\rho_r$ of $\fA$ with the property that
\[
\pi_\mu(s)\mapsto \rho_1(s)\oplus \rho_2(s)\oplus \ldots \oplus\rho_r(s), \quad s\in \fA
\]
is a $*$-isomorphism. Hence
\[
\|\pi_\mu^{(n)}(T)\|=\max_{1\leq j\leq r} \|\rho_j^{(n)}(T)\|.
\]
Thus, there is $1\leq k\leq r$ such that
\[
 \|\rho_k^{(n)}(T)\|\geq \delta \|\pi^{(n)}(T)\|
\]
whence
\[
 \|\rho_k^{(n)}(T)\|\geq \delta(1+\eps) \|\sigma^{(n)}(T)\|, \quad \sigma\in \W_\pi.
\]
Since $\delta(1+\eps)>1$, we infer that $[\rho_k]\notin \W_\pi$ whence $\rho_k$ is weakly equivalent to $\pi$. Because $\rho_k$ is finite-dimensional, so must be $\pi$.
\end{proof}

In the notation of the previous result, we see that $\W_\pi\subset \U_\pi$, so Theorem \ref{T:peakingRFD} applies in particular to strongly peaking representations.
We now obtain a consequence which addresses Question \ref{Q:main}.

\begin{corollary}\label{C:peakingRFDfd}
Let $\A$ be a unital operator algebra such that $\rC_e^*(\A)$ is RFD. Let $\pi$ be a strongly peaking $*$-representation for $\rC^*_e(\A)$.  Then, $\pi$ is a finite-dimensional boundary representation for $\A$.
\end{corollary}
\begin{proof}
Simply combine Theorem  \ref{T:peakingrep}, Lemma \ref{L:unifpeakC*env} and Theorem \ref{T:peakingRFD}.
\end{proof}

In view of Question \ref{Q:main} and of the previous result, it is worthwhile to track down strongly peaking representations among the finite-dimensional ones. We record next a useful observation: a finite-dimensional irreducible $*$-representation of a $\rC^*$-algebra $\fA$ is necessarily strongly peaking  when its support projection belongs to $\fA$ itself (and not only to the larger bidual $\fA^{**}$).

\begin{lemma}\label{L:clopenproj}
Let $\fA$ be a $\rC^*$-algebra. Let $\pi$ be a finite-dimensional irreducible $*$-representation of $\fA$ with support projection $\fs_\pi$ lying in $\fA$. Then, $\fs_\pi$ peaks strongly at $\pi$.
\end{lemma}
\begin{proof}
 Let $\sigma$ be an irreducible $*$-representation of $\fA$ which is not unitary  equivalent to $\pi$. By Lemma \ref{L:fdimweakcontain}, we see that $\ker \pi$ is not contained in $\ker \sigma$, whence $\sigma(\fs_\pi)\neq I$ as $(I-\fs_\pi)\fA^{**}=\ker \pi^{**}$. Since $\sigma$ is irreducible and $\fs_\pi$ is central, this forces $\sigma(\fs_\pi)=0$. We conclude that
\[
\|\pi(\fs_\pi)\|=1>0=\sup_{[\sigma]\in \U_\pi}\|\sigma(\fs_\pi)\|
\]
so $\fs_\pi$ peaks strongly at $\pi$.
\end{proof}

The requirement that the support projection lies in $\fA$ is equivalent to it being simultaneously closed and open \cite[Theorem II.18]{akemann1969} in the non-commutative topology discussed in Subsection \ref{SS:suppproj}. 

We also note that there are known examples of unital operator algebras $\A$ and of injective, infinite-dimensional, irreducible $*$-representations $\pi$ of $\rC^*_e(\A)$ that are not boundary representations for $\A$ \cite[Example 6.6.3]{DK2019}. In this case, $\fs_\pi=I\in \rC^*_e(\A)$, so in light of Theorem \ref{T:peakingrep} and Lemma \ref{L:unifpeakC*env}, we see that the statement of Lemma \ref{L:clopenproj} cannot be strengthened to cover infinite-dimensional representations.
 
 We now give an elementary application.

\begin{corollary}\label{C:centralproj}
Let $\A$ be a unital operator algebra such that $\rC^*_e(\A)$ contains a non-zero finite-dimensional closed two-sided ideal. Then, there is a finite-dimensional strongly peaking boundary representation for $\A$.
\end{corollary}
\begin{proof}
Let $\fJ\subset \rC^*_e(\A)$ be a non-zero finite-dimensional closed two-sided ideal. Let $\fJ'$ be a minimal non-zero closed two-sided ideal in $\fJ$. There is a central projection $e\in \rC^*_e(\A)$ such that $\fJ=\rC^*_e(\A)e$. Thus
\[
\rC^*_e(\A) \fJ'=\rC^*_e(\A) e\fJ'=\fJ \fJ'\subset \fJ' \qand  \fJ'\rC^*_e(\A)=\fJ' e\rC^*_e(\A)=\fJ' \fJ\subset \fJ' .
\]
We infer that $\fJ'$ is a minimal  non-zero closed two-sided ideal of $\rC^*_e(\A)$, which is still finite-dimensional. Therefore, there is a central projection $z\in \rC^*_e(\A)$ such that $\fJ'=\rC^*_e(\A)z$.
Consider the unital $*$-representation $\pi:\rC^*_e(\A)\to \fJ'$ defined as 
\[
\pi(t)=tz, \quad t\in \rC^*_e(\A).
\]
By Lemma \ref{L:suppprojmult} we obtain that $\fs_\pi=z$. Now, $\fJ'$ is a simple finite-dimensional $\rC^*$-algebra so there is a positive integer $n$ and a $*$-isomorphism $\sigma: \fJ'\to \bM_n$. Then, the support projection for the irreducible finite-dimensional $*$-representation $\sigma\circ \pi$ is simply $z$. By Lemma  \ref{L:clopenproj}, we see that $z$ peaks strongly at $\sigma\circ \pi$. Therefore, $\sigma\circ \pi$ is a boundary representation for $\A$ by Theorem \ref{T:peakingrep}. Moreover, $\sigma\circ \pi$ is a strongly peaking representation for $\A$ by virtue of Lemma \ref{L:unifpeakC*env}.
\end{proof}

 The next result gives a sufficient condition for a support projection to be closed and open. Given a bounded linear operator $T$ on a Hilbert space, we denote its spectrum by $\spec(T)$.

\begin{theorem}\label{T:specpeak}
Let $\fA$ be a unital $\rC^*$-algebra and let $\Delta$ be a set of irreducible $*$-representations of $\fA$ which is separating for $\fA$. Assume that there is a positive element $t\in \fA$ and a finite-dimensional $*$-representation $\pi\in \Delta$ with the property 
\[
\min \spec(\pi(t))> \sup_{\sigma\in \Delta, \sigma\not\sim \pi} \|\sigma(t)\|.
\]
Then, the support projection of $\pi$ lies in $\fA$ and peaks strongly at $\pi$.
\end{theorem}
\begin{proof}
Upon scaling if necessary, we may assume that $\|\pi(t)\|=1$. Put
\[
r=\sup_{\sigma\in \Delta, \sigma\not\sim \pi} \|\sigma(t)\|
\]
and
\[
R=\min\{\lambda:\lambda\in \spec(\pi(t))\}.
\]
Thus, we have $0\leq r<R\leq1 $. 
Choose a continuous function $f: [0,1] \to [0,1]$ with the property that $f=1$ on the interval $[R,1]$ and $f=0$ on the interval $[0,r]$. If $\sigma$ is an irreducible $*$-representation of $\fA$ with $\sigma\sim \pi$, then $\spec(\pi(t))=\spec(\sigma(t))$. Since $\Delta$ is assumed to be separating for $\fA$, we conclude that $\spec(t)$ coincides with the closure of
\[
\spec(\pi(t))\cup \bigcup_{\sigma\in \Delta,\sigma\not\sim \pi} \spec(\sigma(t)).
\]
Thus, 
\[
\spec(t)\subset [0,r]\cup [R,1].
\]
If we let $p=f(t)$, then we see that $p$ is a self-adjoint projection in $\fA$. Furthermore, we have
\[
\pi(p)=\pi(f(t))=f(\pi(t))=I
\]
since $f=1$ on $\spec(\pi(t))\subset [R,1]$. Likewise, if $\sigma\in \Delta$ is such that $\sigma\not\sim\pi$, then $f=0$ on $\spec(\sigma(t))\subset [0,r]$ so
\[
\sigma(p)=\sigma(f(t))=f(\sigma(t))=0.
\]
We also infer that
\[
\pi(ps-sp)=\pi(s)-\pi(s)=0
\] 
and
\[
\sigma(ps-sp)=0-0=0
\]
for every $\sigma\in \Delta,\sigma\not\sim \pi$ and  every $s\in \fA$.  Hence, using that $\Delta$ is separating for $\fA$, we conclude that $p$ is a non-zero central projection and $\fA p\subset\bigcap_{\sigma\in \Delta,\sigma\not\sim \pi}\ker \sigma$. Observe next that
\[
\ker \left(  \bigoplus_{\sigma\in \Delta, \sigma\sim\pi}\sigma\right)=\ker \pi.
\]
Using once again that $\Delta$ is separating for $\fA$, we see that  $\fA p$ is $*$-isomorphic to 
\begin{align*}
\left( \bigoplus_{\sigma\in \Delta}\sigma\right)(\fA p) &\cong\left( \bigoplus_{\sigma\in \Delta, \sigma\sim\pi}\sigma\right)(\fA p).
\end{align*}
By Lemma \ref{L:fdimweakcontain}, we see that every $\sigma\in \Delta$ with $\sigma\sim \pi$ is in fact unitarily equivalent to $\pi$, whence $\fA p$ is $*$-isomorphic to $\pi(\fA p)$, and in particular
%
 is a non-zero finite-dimensional closed two-sided ideal. 

Since $\pi(p)=I$ we have $\fs_\pi\leq p$. In particular, $\fs_\pi$ lies in $\fA^{**}p$. This last set coincides with the weak-$*$ closure of $\fA p$ inside of $\fA^{**}$. But $\fA p$ is finite-dimensional and hence weak-$*$-closed already, so that $\fs_\pi\in \fA p\subset  \fA$. We conclude from Lemma \ref{L:clopenproj} that $\fs_\pi$ peaks strongly at $\pi$.
\end{proof}

We close this section with another application of our ideas. This example is related to \cite[Theorem 3.9]{DP2020}, which is obtained by other means.

 \begin{example}\label{E:specpeak}
Let $\A$ be a unital operator algebra with the property that $\rC^*_e(\A)$ is RFD. We may thus assume that there is a set $\{r_\lambda:\lambda\in \Lambda\}$ of positive integers with the property that 
 \[
 \A\subset \rC^*_e(\A)\subset  \prod_{\lambda\in \Lambda} \bM_{r_{\lambda}}
 \] 
 and that each natural projection $\pi_\lambda:\rC_e^*(\A)\to \bM_{r_\lambda}$ is surjective. Hence, each $\pi_\lambda$ is a finite-dimensional irreducible $*$-representation, and the set $\Delta=\{\pi_\lambda:\lambda\in \Lambda\}$ is separating for $\rC_e^*(\A)$. 
  Assume also that there is a positive element $t=(t_\lambda)_{\lambda\in \Lambda}\in \rC_e^*(\A)$ such that there is $\lambda_0\in \Lambda$ with 
 \[
 \min \spec(t_{\lambda_0})> \sup_{\lambda\neq \lambda_0}\|t_\lambda\|.
 \]
 Then, we may apply Theorem \ref{T:specpeak} to conclude that the support projection of  $\pi_{\lambda_0}$ belongs to $\rC_e^*(\A)$ and peaks strongly at $\pi_{\lambda_0}$. By Lemma \ref{L:unifpeakC*env}, we see that $\pi_{\lambda_0}$ is a strongly peaking representation for $\A$. In particular, we see that $\pi_{\lambda_0}$ is a finite-dimensional boundary representation for $\A$ by Theorem \ref{T:peakingrep}. 
  \qed
 \end{example}


\section{RFD $\rC^*$-envelopes and $\rC^*$-liminality}\label{S:C*lim}

In this section, we study unital operator spaces whose $\rC^*$-envelopes are assumed to admit a large supply of a finite-dimensional boundary representations. Our first result is a characterization of the residual finite-dimensionality of the $\rC^*$-envelope, in the spirit of Theorem \ref{T:EL}.

\begin{theorem}\label{T:ELUEP}
Let $\M$ be a unital operator space. Consider the following statements.
\begin{enumerate}[{\rm (i)}]
\item The algebra $\rC^*_e(\M)$ is RFD.
\item Every boundary representation for $\M$ on $\rC^*_e(\M)$ admits a finite-dimensional approximation.
\item Every boundary representation for $\M$ on $\rC^*_e(\M)$  admits a finite-dimensional approximation consisting of $*$-representations with the essential unique extension property with respect to $\M$.
\item The set of unitary equivalence classes of finite-dimensional boundary representations for $\M$ is dense in $\widehat{\rC^*_e(\M)}$.
\item There is a set of finite-dimensional boundary representations for $\M$ which is separating for $\rC^*_e(\M)$.

\end{enumerate}
Then, we have that   {\rm(v)} $\Longleftrightarrow$ {\rm(iv)} $\Longleftrightarrow$ {\rm(iii)} $\Longrightarrow$ {\rm(ii)} $\Longleftrightarrow$ {\rm(i)}.
\end{theorem}
\begin{proof}

{\rm(v)} $\Longrightarrow$ {\rm(iv)}: This is an immediate consequence of the definition of the topology on the spectrum of $\rC^*_e(\M)$. 

{\rm(iv)} $\Longrightarrow$ {\rm(iii)}: This argument closely follows the ideas of \cite{EL1992}.  Let $\F$ denote the set of pure states on $\rC^*_e(\M)$ whose GNS representation is a finite-dimensional boundary representation for $\M$. By assumption, we may invoke Lemma \ref{L:densestaterep} to find that $\F$ is weak-$*$ dense in  the pure states on $\rC^*_e(\M)$. 

Let $\pi:\rC^*_e(\M)\to B(\H)$ be a boundary representation for $\M$. For our purposes, it is no loss of generality  to assume that $\H$ is infinite-dimensional. Let $\xi\in \H$ be a unit vector. Let $\phi:\rC^*_e(\M)\to\bC$ be the state defined as
\[
\phi(t)=\langle \pi(t)\xi,\xi \rangle, \quad t\in \rC^*_e(\M).
\]
Since $\pi$ is irreducible, $\phi$ is a pure state. By the previous paragraph, there is a net $(\psi_\alpha)$ of states in $\F$ that converges to $\phi$ in the weak-$*$ topology. Let $\sigma_\alpha:\rC^*_e(\M)\to B(\H_\alpha)$ be the GNS representation of $\psi_\alpha$. Then, $\sigma_\alpha$ is a finite-dimensional boundary representation for $\M$. Arguing as in the proof of \cite[Theorem 2.4]{EL1992}, we find an isometry $V_\alpha:\H_\alpha\to \H$ such that 
\[
\lim_\alpha V_\alpha \sigma_\alpha(t) V_\alpha^* h=\pi(t)h, \quad t\in \rC_e^*(\M), h\in \H.
\]
By Lemma \ref{L:UEPisom}, we see that the net $(V_\alpha \sigma_\alpha(\cdot)V_\alpha^*)_\alpha$ is a finite-dimensional approximation for $\pi$ consisting of $*$-representations with the essential unique extension property with respect to $\M$.

{\rm(iii)} $\Longrightarrow$ {\rm(v)}: Let $t\in \rC^*_e(\M)$ be a non-zero element. By Lemma \ref{L:bdrysepequiv}, there is a boundary representation  $\pi:\rC^*_e(\M)\to B(\H_\pi)$ for $\M$ with $\pi(t)\neq 0$. Let $\xi\in \H_\pi$ be a unit vector such that $\pi(t)\xi\neq 0$. By assumption, there is a finite-dimensional approximation
\[
\pi_\lambda:\rC^*_e(\M)\to B(\H_\pi), \quad \lambda\in\Lambda
\] 
for $\pi$ consisting of $*$-representations with the essential unique extension property with respect to $\M$. For each $\lambda\in \Lambda$, let $\E_\lambda=\pi_\lambda(\rC^*_e(\M))\H_\pi$ and note that this is a finite-dimensional subspace. Let $\rho_\lambda:\rC^*_e(\M)\to B(\E_\lambda)$ be the unital finite-dimensional $*$-representation defined as
\[
\rho_\lambda(t)=\pi_\lambda(t)|_{\E_\lambda}, \quad t\in \rC^*_e(\M).
\]
Then, $\rho_\lambda$ has the unique extension property with respect to $\M$. We see that
\[
\lim_{\lambda\in\Lambda}\rho_\lambda(t)P_{\E_\lambda}\xi=\lim_{\lambda\in\Lambda}\pi_\lambda(t)\xi=\pi(t)\xi\neq 0
\]
whence there is $\lambda\in \Lambda$ such that $\rho_\lambda(t)\neq 0$. Since $\rho_\lambda$ is finite-dimensional, it can be written as the direct sum of boundary representations for $\M$ by Lemma \ref{L:directsumUEP}. In particular, there is a boundary representation $\beta$ for $\M$ on $\rC^*_e(\M)$ such that $\beta(t)\neq 0$. We conclude that there is a set of finite-dimensional boundary representations for $\M$ which is separating for $\rC^*_e(\M)$.

{\rm(iii)} $\Longrightarrow$ {\rm(ii)}: Trivial.

{\rm(ii)} $\Longrightarrow$ {\rm(i)}:  
Let $t\in \rC^*_e(\M)$ be a non-zero element. By Lemma \ref{L:bdrysepequiv}, there is a boundary representation  $\pi:\rC^*_e(\M)\to B(\H_\pi)$ for $\M$ with $\pi(t)\neq 0$. Let $\xi\in \H_\pi$ be a unit vector such that $\pi(t)\xi\neq 0$. By assumption, there is a finite-dimensional approximation
\[
\pi_\lambda:\rC^*_e(\M)\to B(\H_\pi), \quad \lambda\in\Lambda
\] 
for $\pi$.  For each $\lambda\in \Lambda$, let $\E_\lambda=\pi_\lambda(\rC^*_e(\M))\H_\pi$ and note that this is a finite-dimensional subspace. Let $\rho_\lambda:\rC^*_e(\M)\to B(\E_\lambda)$ be the unital finite-dimensional $*$-representation defined as
\[
\rho_\lambda(t)=\pi_\lambda(t)|_{\E_\lambda}, \quad t\in \rC^*_e(\M).
\]
We see that
\[
\lim_{\lambda\in\Lambda}\rho_\lambda(t)P_{\E_\lambda}\xi=\lim_{\lambda\in\Lambda}\pi_\lambda(t)\xi=\pi(t)\xi\neq 0
\]
whence there is $\lambda\in \Lambda$ such that $\rho_\lambda(t)\neq 0$.  We conclude that there is a set of finite-dimensional $*$-representations of $\rC^*_e(\M)$ which is separating for $\rC^*_e(\M)$.

{\rm(i)} $\Longrightarrow$ {\rm(ii)}:  This is an immediate consequence of Theorem \ref{T:EL}.

\end{proof}

Recall that Question \ref{Q:main} asks whether $\rC^*_e(\A)$ being RFD forces the existence of even a \emph{single} finite-dimensional boundary representation. This offers some perspective regarding the potential validity of the implication {\rm(i)} $\Longrightarrow$ {\rm(v)} above.

We make a few related remarks.
Inside the spectrum of $\rC^*_e(\M)$, let $\D$ denote the set of unitary equivalence classes of finite-dimensional irreducible $*$-representations, and let $\B$ denote the set of unitary equivalence classes of boundary representations for $\M$. It follows from Theorems \ref{T:EL} and \ref{T:density} that both these subsets are dense, but we do not know if these two large sets intersect non-trivially. Finer topological properties may be useful in resolving this. Indeed, inspired by the statement of the Baire Category theorem, it is natural to wonder if the sets are of type $G_\delta$ for instance. This seems unlikely, as  \cite[Proposition 3.6.3]{dixmier 1977} implies that $\D$ is in fact of type $F_\sigma$ in $\widehat{\rC^*_e(\M)}$.

As discussed in Subsection \ref{SS:RFD}, there are examples of unital $\rC^*$-algebras that are RFD but not liminal. In particular, the set $\D$ defined above is a proper subset of the spectrum in this case. Moreover, the set $\B$ defined above can also be a proper subset of the spectrum, even in the classical situation of a uniform algebra: there are known examples of uniform algebras on compact metric spaces where the Choquet boundary is a proper (albeit dense) subset of the Shilov boundary (see for instance \cite[page 42]{phelps2001}).

\subsection{$\rC^*$-liminality}\label{SS:C*lim}

In this subsection, we approach our main problem from another angle. Herein, we study unital operator spaces with the property that \emph{all} boundary representations are finite-dimensional. The hope is that a thorough understanding of this extremal situation may shed some light on Question \ref{Q:main}.

A unital operator space $\M$ is said to be \emph{$\rC^*$-liminal} if every boundary representation for $\M$ on $\rC^*_e(\M)$ is finite-dimensional. 
When $\M$ is a $\rC^*$-algebra, a boundary representation is simply an irreducible $*$-representation, so this notion coincides with the usual one introduced in Subsection \ref{SS:RFD}. We remark that a certain related notion of what one may call ``$\rC^*$-subhomogeneity" was studied recently in \cite{AHMR2020subh} in a different context.
 
 It follows from Theorem \ref{T:CS} that a unital $\rC^*$-algebra is liminal if and only if it is NFD. Along similar lines, Theorem \ref{T:bdryexist} implies that a $\rC^*$-liminal unital operator algebra is necessarily NFD. But the converse is false in general. We illustrate this below (Example \ref{E:fdimnotNFD}) with a rather pathological example of a finite-dimensional operator algebra  with a unique boundary representation (up to unitary equivalence) and such that this boundary representation is infinite-dimensional. As preparation, we need a basic fact.

\begin{lemma}\label{L:uppertriang}
Let $\M\subset B(\H)$ be a finite-dimensional unital operator space. Consider the set $\B_\M\subset \bM_2(B(\H))$ consisting of those elements of the form
\[
\begin{bmatrix}
\lambda I & x \\
0 & \mu I
\end{bmatrix}
\]
where $\lambda,\mu\in \bC, x\in \M$. Then, $\B_\M$ is a finite-dimensional unital operator algebra and
$
\rC_e^*(\B_\M)\cong \bM_2(\rC_e^*(\M)).
$
Moreover, every boundary representation for $\B_\M$ on $\bM_2(\rC^*_e(\M))$  is  unitarily equivalent to one of the form $\beta^{(2)}$ for some boundary representation $\beta$ for $\M$ on $\rC^*_e(\M)$.
\end{lemma}
\begin{proof}
A routine calculation shows that $\B_\M$ is a finite-dimensional unital operator algebra and that $\rC^*(\B_\M)=\bM_2(\rC^*(\M))$. By \cite[Lemma 3.1]{CR2018rfd} we find
$
\rC_e^*(\B_\M)\cong \bM_2(\rC_e^*(\M)).
$ 
Let $\pi:\bM_2(\rC^*_e(\M))\to B(\H_\pi)$ be a boundary representation for $\B_\M$. Then, there is a boundary representation $\beta:\rC^*_e(\M)\to B(\H_\beta)$ for $\M$ with the property that $\pi$ is unitarily equivalent to $\beta^{(2)}$ by the main result of \cite{hopenwasser1973}. 
\end{proof}

We can now give the announced example of a finite-dimensional operator algebra with no finite-dimensional boundary representation.

 \begin{example}\label{E:fdimnotNFD}
 By \cite[Theorem 6.2]{CH2018}, there exists a Hilbert space $\H$ and an operator $T\in B(\H)$ with the property that the norm-closed unital algebra $\A$ generated by $T$ is infinite-dimensional and has the identity representation as its only boundary representation (up to unitary equivalence). Let $\M=\bC I+\bC T$. Then, $\rC^*(\M)=\rC^*(\A)$ and it is easily verified that a boundary representation for $\M$ on $\rC^*(\M)$ is necessarily a boundary representation for $\A$ as well. Thus, the identity representation of $\rC^*(\M)$ is the only boundary representation for $\M$, up to unitary equivalence.  In particular, $\rC^*_e(\M)\cong \rC^*(\M)$.
 
 Let $\B_\M\subset \bM_2(B(\H))$ be the finite-dimensional unital operator algebra as in Lemma \ref{L:uppertriang}. Then, $\rC^*_e(\B_\M)\cong \bM_2(\rC^*(\M))$ and up to unitary equivalence, the only boundary representation for $\B_\M$ on $\bM_2(\rC^*(\M))$ is $\id^{(2)}$, which is infinite-dimensional. Hence $\B_\M$ is not $\rC^*$-liminal. Note also that \cite[Theorem 3.5]{CR2018rfd} implies that $\B_\M$ is NFD. 
   \qed
 \end{example}
 
 In fact, things can be even worse, and the $\rC^*$-envelope of a finite-dimensional unital operator algebra can have no finite-dimensional $*$-representations whatsoever.
 
  \begin{example}\label{E:Cuntz}
 Let $\H$ be a Hilbert space and let $U,V\in B(\H)$ be isometries such that $UU^*+VV^*=I$. Let $\M=\bC I+\bC U+\bC V$. Then, $\rC^*(\M)$ is $*$-isomorphic to the Cuntz algebra $\O_2$, which is infinite-dimensional and simple \cite[Theorem V.4.7]{davidson1996}.  In particular, the Shilov ideal of $\M$ in $\rC^*(\M)$ is trivial, so $\rC^*_e(\M)\cong \rC^*(\M)=\O_2$. 
 
Let $\B_\M\subset \bM_2(B(\H))$ be the finite-dimensional unital operator algebra as in Lemma \ref{L:uppertriang}. We find  $\rC^*_e(\B_\M)\cong \bM_2(\O_2)$ and $\rC^*_e(\B_\M)$ has no finite-dimensional $*$-representation.
   \qed
 \end{example}
 
 As discussed in the introduction, the previous pair of examples illustrates some of the subtleties inherent to Question \ref{Q:main}. In particular, the existence of finite-dimensional boundary representations appears to lie much deeper than the mere finite-dimensionality of the original operator algebra. It is conceivable that Question \ref{Q:main} may thus have a negative answer in general. An operator algebra witnessing this pathology would need to have an RFD $\rC^*$-envelope that is not liminal.  Unfortunately, we do not have a large supply of examples of such $\rC^*$-algebras. A standard one is the full group $\rC^*$-algebra of the free group on two generators, as mentioned in Subsection \ref{SS:RFD}. Generally speaking  however, natural examples from group theory tend to admit finite-dimensional boundary representations, as we show next. 

\begin{example}\label{E:groups}
Let $G$ be a discrete group that is not virtually abelian. As explained in \cite{CS2019}, it then follows from \cite{thoma1964},\cite{thoma1968} that the full group $\rC^*$-algebra $\rC^*(G)$ is not liminal. Fix some generating subset $S$ for $G$. Let $\A\subset \rC^*(G)$ be the unital subalgebra generated by the unitaries  corresponding to the elements of $S$. It is readily seen that $\rC^*(\A)=\rC^*(G)$. Moreover, by definition we see that $\A$ contains unitaries that generate $\rC^*(G)$ as a $\rC^*$-algebra, so it follows that every irreducible $*$-representation of $\rC^*(G)$ is a boundary representation for $\A$ (see for instance \cite[Lemma 5.5]{kavruk2014}). Consequently $\rC^*_e(\A)\cong \rC^*(G)$. But $\rC^*(G)$ always admits the character corresponding to the trivial representation of $G$, which is then a finite-dimensional boundary representation for $\A$.
\qed
\end{example}

Next, we exhibit another method for constructing a unital operator algebra whose $\rC^*$-envelope is RFD and not liminal, while admitting many finite-dimensional boundary representations.

\begin{example}\label{E:tensorprodC(X)}
Let $X$ be a compact metric space and let $\A\subset C(X)$ be a uniform algebra. Assume that $X$ is the Shilov boundary of $\A$. Let $\fR$ be a unital RFD $\rC^*$-algebra which is not liminal.  It follows from \cite{hopenwasser1978} that the boundary representations for $\A\otimes_{\min} \fR$ on $C(X)\otimes_{\min} \fR$ are precisely those of the form $\eps_\xi\otimes \pi$ where $\pi$ is an irreducible $*$-representation of $\fR$ and $\eps_\xi$ is the character on $C(X)$ of evaluation at a point $\xi$ in the Choquet boundary of $\A$.
In particular, because $\fR$ is RFD, we conclude that there are many finite-dimensional boundary representations for $\A\otimes_{\min} \fR$ on $C(X)\otimes_{\min} \fR$.

Next, we note that $C(X)\otimes_{\min}\fR\cong C(X,\fR)$ \cite[Proposition 12.5]{paulsen2002}. From this identification, it is readily seen that the Shilov ideal of $\A\otimes_{\min} \fR$ is trivial, so that 
\[
\rC^*_e(\A\otimes_{\min} \fR)\cong C(X)\otimes_{\min}\fR\cong C(X,\fR).
\]
In particular, we see that the $\rC^*$-envelope is RFD and not liminal. 
\qed
\end{example}

Because boundary representations always factor through the $\rC^*$-envelope, a unital operator space $\M$ is $\rC^*$-liminal whenever $\rC^*_e(\M)$ is liminal in the classical  sense. In the next development we obtain a sort of partial converse. First, we give a characterization of $\rC^*$-liminality in terms of certain pure linear maps.

Given a unital operator space $\M$ and a unital completely contractive map $\psi:\M\to B(\H_\psi)$, we say that $\psi$ is \emph{finite-dimensional} if there exists a finite-dimensional unital $*$-representation $\sigma:\rC_e^*(\M)\to B(\H_\sigma)$ and an isometry $V:\H_\psi\to \H_\sigma$ such that
\[
\psi(a)=V^*\sigma(a)V, \quad a\in \M.
\]
Note in particular that $\H_\psi$ must be finite-dimensional in this case, so only matrix states can ever be finite-dimensional in this sense.

\begin{proposition}\label{P:fdimbdrypure}
Let $\M$ be a unital operator space. Then, $\M$ is $\rC^*$-liminal if and only if all pure unital completely contractive maps on $\M$ are finite-dimensional.
\end{proposition}
\begin{proof}
Assume first that $\M$ is $\rC^*$-liminal. Let $\psi:\M\to B(\H_\psi)$ be a pure unital completely contractive map. 
Applying \cite[Theorem 2.4]{davidsonkennedy2015}, we see that $\psi$ dilates to a boundary representation for $\M$ on $\rC^*_e(\M)$, and thus $\psi$ is finite-dimensional.

Conversely, assume that all pure unital completely contractive maps on $\M$ are finite-dimensional. Let $\pi:\rC_e^*(\M)\to B(\H_\pi)$ be a boundary representation for $\M$. 
 The restriction $\pi|_\M$ is a pure unital completely contractive map \cite[Proposition 2.12]{FT2020}, and thus it is finite-dimensional by assumption. In particular, $\H_\pi$ must be finite-dimensional. 
\end{proof}

We record a consequence of the previous result.

\begin{corollary}\label{C:fdimbdryprim}
Let $\M$ be a $\rC^*$-liminal unital operator space and let $\fB\subset \M$ be a $\rC^*$-algebra with the same unit as $\M$. Then, $\fB$ is liminal.
\end{corollary}
\begin{proof}
Let  $\pi$ be an irreducible $*$-representation of $\fB$. By \cite[Theorem 1.4.2]{arveson1969}, we see that $\pi$ is a pure unital completely positive map on $\fB$, and thus it admits a pure extension to a  unital completely contractive map $\pi'$ on $\M$ \cite[Corollary 2.3]{kleski2014bdry}. By Proposition \ref{P:fdimbdrypure}, we see that $\pi'$ is finite-dimensional, so that $\pi$ must act on a finite-dimensional space.
\end{proof}

For $\rC^*$-algebras, the statement of the previous corollary is well known and follows for instance from  Theorem \ref{T:CS}. 

Let $\M$ be a unital operator space. We say that a unital completely contractive map $\psi:\M\to B(\H_\psi)$ is \emph{locally finite-dimensional} if whenever $\N\subset \M$ is a unital finite-dimensional subspace, there exist a finite-dimensional unital $*$-representation $\sigma:\rC_e^*(\M)\to B(\H_\sigma)$ and an isometry $V:\H_\psi\to \H_\sigma$ such that
\[
\psi(a)=V^*\sigma(a)V, \quad a\in \N.
\]
Note once again that this forces $\H_\psi$ to be finite-dimensional. The next result is a minor refinement of a recent result of Hartz and Lupini \cite[Theorem 1.5]{HL2019}. Roughly speaking, it says that all matrix states of a $\rC^*$-liminal operator space are locally finite-dimensional. Thus, locally we can remove the purity requirement of Proposition \ref{P:fdimbdrypure}.

\begin{theorem}\label{T:fdimbdryfdimstates}
Let $\M$ be a $\rC^*$-liminal unital operator space. Then, every matrix state of $\M$ is locally finite-dimensional.
\end{theorem}
\begin{proof}
Throughout the proof, we let $\fT$ denote the operator system generated by $\M$ inside of $\rC^*_e(\M)$. Let $\N\subset \M$ be a unital finite-dimensional subspace and let $\fS$ denote the operator system generated by $\N$ inside of $\rC^*_e(\M)$. Note that $\fS$ is still finite-dimensional.

We start with a preliminary observation. Let $\omega:\fS\to B(\H_\omega)$ be a pure unital completely positive map. It follows from   \cite[Corollary 2.3]{kleski2014bdry} that there is a pure unital completely positive map $\Omega:\fT\to  B(\H_\omega)$ extending $\omega$. In light of \cite[Theorem 2.4]{davidsonkennedy2015}, we know that $\Omega$ can be dilated to a boundary representation for $\M$ on $\rC_e^*(\M)$. By assumption, this means that $\Omega$ (and hence $\omega$) can be dilated to a unital finite-dimensional $*$-representation of $\rC^*_e(\M)$.

 We now turn to proving the desired statement. Let $\psi:\N\to \bM_r$ be a matrix state and let $\Psi$ denote its unique unital completely positive extension to $\fS$. Invoke \cite[Theorem 2.9]{HL2019} (see also \cite[Section 3]{HL2019})  to conclude that $\Psi$ is a finite matrix convex combination of matrix extreme points of the matrix state space of $\fS$. In turn, by \cite[Theorem B]{farenick2000}, we infer that the matrix extreme points among matrix states are pure matrix states on $\fS$. As shown in the previous paragraph, all pure matrix states on $\fS$ can be dilated to a finite-dimensional $*$-representation of $\rC^*_e(\M)$.  Arguing now as in the proof of \cite[Lemma 3.2]{HL2019}, we see that $\Psi$ dilates to a finite-dimensional unital $*$-representation of $\rC^*_e(\M)$. Thus, there is a finite-dimensional  unital $*$-representation $\sigma$ of $\rC^*_e(\M)$ and an isometry $V$ with the property that
\[
\Psi(s)=V^* \sigma(s)V, \quad s\in \fS. 
\]
Restricting to $\N$, this yields the desired statement for $\psi$.
\end{proof}

A few remarks concerning this theorem are in order. First, recall that the motivation behind this result is the question asking whether $\rC^*$-liminality of a unital operator space implies liminality of its $\rC^*$-envelope. Although we do not manage to prove this, Theorem \ref{T:fdimbdryfdimstates} supports this possibility. Indeed, if $\M$ is a $\rC^*$-algebra, then the conclusion is known to imply liminality \cite[Proposition 1.4]{HL2019}. Furthermore, the conclusion of the previous theorem is necessary for the liminality of the $\rC^*$-envelope, by \cite[Theorem 1.5]{HL2019}.

While we do not know if the converse to Theorem \ref{T:fdimbdryfdimstates} holds, we can at least prove the following.

\begin{theorem}\label{T:locfdim}
Let $\M$ be a unital operator space. Assume that every matrix state on $\M$ is locally finite-dimensional. Then, every unital $*$-representation of $\rC_e^*(\M)$ with the unique extension property with respect to $\M$ admits a finite-dimensional approximation. In particular, $\rC^*_e(\M)$ is RFD.
\end{theorem}
\begin{proof}
Let $\beta:\rC_e^*(\M)\to B(\H_\beta)$ be a unital $*$-representation with the unique extension property with respect to $\M$. Without loss of generality, we may assume that $\H_\beta$ is infinite-dimensional, for otherwise the desired result is trivial. Let $\Lambda$ denote the collection of pairs $(\N,\X)$ where $\N\subset \M$ is a unital finite-dimensional subspace and $\X\subset \H_\beta$ is a finite-dimensional subspace. Then, $\Lambda$ is a directed set. Let $\lambda=(\N,\X)\in \Lambda$. The map 
\[
a\mapsto P_{\X}\beta(a)|_\X, \quad a\in \M
\]
is a matrix state on $\M$. By assumption, there is a finite-dimensional unital $*$-representation $\sigma_\lambda:\rC_e^*(\M)\to B(\K_\lambda)$ and an isometry $V_\lambda:\X\to \K_\lambda$ with the property that
\[
P_{\X}\beta(a)|_\X=V_\lambda^* \sigma_\lambda(a) V_\lambda, \quad a\in \N.
\]
Since $\H_\beta$ is assumed to be infinite-dimensional, there is an isometry $W_\lambda:\K_\lambda\to \H_\beta$. Define $\rho_\lambda:\rC_e^*(\M)\to B(\H_\beta)$ as
\[
\rho_\lambda(t)=W_\lambda \sigma_\lambda(t) W_\lambda^*, \quad t\in \rC_e^*(\M). 
\]
This is a $*$-representation and $\rho_\lambda(\rC_e^*(\M))\H_\beta\subset W_\lambda \K_\lambda$ so that $\rho_\lambda(\rC_e^*(\M))\H_\beta$ is finite-dimensional. Again, since $\H_\beta$ is assumed to be infinite-dimensional, we can extend the isometry $W_\lambda V_\lambda: \X\to \H_\beta$ to a unitary $U_\lambda$ on $\H_\beta$. Note that
\[
U_\lambda|_\X=W_\lambda V_\lambda \qand P_\X U_\lambda^*=V_\lambda^* W_\lambda^*
\]
so we find
\begin{equation}\label{Eq:dilation}
P_{\X}\beta(a)|_\X=P_{\X}U_\lambda^* \rho_\lambda(a) U_\lambda|_{\X}, \quad a\in \N.
\end{equation}

Next, for each $\lambda\in \Lambda$ we define a $*$-representation $\pi_\lambda:\rC_e^*(\M)\to B(\H_\beta)$ as
\[
\pi_\lambda(t)=U_\lambda^* \rho_\lambda(t) U_\lambda, \quad t\in \rC_e^*(\M).
\]
We find
\begin{align*}
\pi_\lambda(\rC_e^*(\M))\H_\beta&= U_\lambda^* \rho_\lambda(\rC_e^*(\M))\H_\beta
\end{align*}
which is again finite-dimensional.  Let  $\psi:\rC_e^*(\M)\to B(\H_\beta)$ be the limit of a subnet $(\pi_{\lambda_\mu})_\mu$ in the pointwise weak-$*$ topology; this exists by compactness \cite[Theorem 7.4]{paulsen2002}. In particular, $\psi$ is a completely positive map. 

Let $a\in\M$ and $\xi,\eta\in \H_\beta$. Let $\mu$ such that $\lambda_\mu=(\N,\X) \in \Lambda$ satisfies $a\in \N$ and $\xi,\eta\in \X$. Then, using Equation \eqref{Eq:dilation} we obtain
\begin{align*}
\langle \beta(a)\xi,\eta\rangle &=\langle P_\X \beta(a)|_\X \xi,\eta \rangle=\langle P_{\X}\pi_{\lambda_\mu}(a)|_{\X} \xi,\eta \rangle=\langle \pi_{\lambda_\mu}(a) \xi,\eta \rangle
\end{align*}
whence
\[
\langle \beta(a)\xi,\eta\rangle=\lim_\mu \langle \pi_{\lambda_\mu}(a) \xi,\eta \rangle=\langle \psi(a)\xi,\eta\rangle.
\]
Therefore, $\psi$ and $\beta$ agree on $\M$. By the unique extension property of $\beta$ with respect to $\M$, we conclude that 
\[
\beta(t)=\psi(t)=\lim_\mu\pi_{\lambda_\mu}(t), \quad t\in \rC_e^*(\M)
\]
where the limit exists in the weak-$*$ topology of $B(\H_\beta)$. By \cite[Paragraph 3.5.2]{dixmier1977}, we conclude that $(\pi_{\lambda_\mu})_\mu$ is a finite-dimensional approximation for $\beta$. 

Finally, we may invoke Theorem \ref{T:ELUEP} to see that $\rC^*_e(\M)$ is RFD. 
\end{proof}

Note that by  Theorem \ref{T:ELUEP}, in order to establish that $\rC^*_e(\A)$ is RFD above, it would have sufficed to produce a finite-dimensional approximation for an arbitrary boundary representation $\beta$. In other words, we could have assumed that $\beta$ was irreducible. Nevertheless, even then it is not clear that a matrix state of the form
\[
a\mapsto P_\X \beta(a)|_\X, \quad a\in \M
\]
would necessarily be pure. This explains why we need to assume that \emph{every} matrix state is locally finite-dimensional, as opposed to simply the pure ones, and emphasizes the relevance of Theorem \ref{T:fdimbdryfdimstates}.

Let us now summarize our findings related to $\rC^*$-liminality.

\begin{corollary}\label{C:C*lim}
Let $\M$ be a unital operator space. Consider the following statements.
\begin{enumerate}[{\rm (i)}]
\item The operator space $\M$ is $\rC^*$-liminal.
\item Every matrix state of $\M$ is locally finite-dimensional.
\item The algebra $\rC^*_e(\M)$ is RFD.
\end{enumerate}
Then, we have  {\rm (i)} $\Rightarrow$ {\rm (ii)} $\Rightarrow$ {\rm (iii)}. If $\M$ is in fact a unital operator algebra that satisfies {\rm (ii)}, then $\M$ must be NFD.
\end{corollary}
\begin{proof}
{\rm (i)} $\Rightarrow$ {\rm (ii)} $\Rightarrow$ {\rm (iii)}: This follows by combining Theorems \ref{T:fdimbdryfdimstates} and \ref{T:locfdim}.

Assume now that $\M$ is a unital operator algebra and every matrix state of $\M$ is locally finite-dimensional. Let $n\in \bN$ and let $A\in \bM_n(\M)$. Write $A=[a_{ij}]$ and let $\N\subset \M$ be the unital finite-dimensional subspace generated by $\{a_{ij}:1\leq i,j\leq n\}$.
By Theorem \ref{T:bdryexist}, there is a matrix state $\psi:\M\to B(\F)$ such that $\|\psi^{(n)}(A)\|=\|A\|$. By assumption, we infer the existence of a finite-dimensional unital $*$-representation $\pi: \rC_e^*(\M) \to B(\H_\pi)$ and of an isometry $V:\F\to \H_\pi$ such that
\[
\psi(b)=V^* \pi(b)V, \quad b\in \N.
\]
In particular, we see that
\[
\psi^{(n)}(A)=V^{*(n)}\pi^{(n)}(A)V^{(n)}
\]
whence $\|\pi^{(n)}(A)\|=\|A\|$. We conclude that $\M$ is indeed NFD.
\end{proof}

By choosing $\M$ to be a unital $\rC^*$-algebra which is RFD but not liminal, we see that (iii) $\not\Rightarrow$ (i) and (iii) $\not\Rightarrow$ (ii) above. Furthermore, Example \ref{E:Cuntz} shows that (ii) can fail for NFD unital operator algebras. As mentioned before, we do not know if (ii) $\Rightarrow$ (i). In \cite{HL2019}, this is shown to hold in the self-adjoint context by using that (i) is equivalent to being NFD for $\rC^*$-algebras, by virtue of Theorem \ref{T:CS}. Unfortunately, that equivalence fails in our generality.

\bibliography{/Users/Raphael/Dropbox/Research/Bibliography/biblio_main}
\bibliographystyle{plain}


\end{document}